\documentclass[leqno,onefignum,onetabnum]{siamltex1213}
\usepackage{bbm}
\usepackage{amsmath,amsfonts,amssymb,enumerate}
\usepackage{color}
\usepackage{lscape}
\usepackage{latexsym}
\usepackage{multirow}
\usepackage{rotating}
\usepackage{threeparttable}
\usepackage{graphicx}
\usepackage{algpseudocode,algorithm}
\usepackage{graphicx}
\usepackage{subfigure}
\usepackage{bigstrut}

\newtheorem{assumption}{Assumption}

\newcommand{\DF}[2]{{\displaystyle\frac{#1}{#2}}}
\newcommand{\innerP}[2]{\langle {#1},{#2} \rangle}
\newcommand{\dist}{\mathrm{dist}}
\newcommand{\dom}{\mathrm{dom}}
\newcommand{\lev}{\mathrm{lev}}

\title{First-order algorithms for a class of fractional optimization problems
	\thanks{This research is supported in part by the Natural Science Foundation of China under grants 11971499 and 11701189,
	and by Guangdong Provincial Key Laboratory of Computational Science at Sun Yat-sen University (2020B1212060032).}
}

\author{
	Na Zhang \thanks{Department of Applied Mathematics, College of Mathematics and Informatics, South China Agricultural University, Guangzhou 510642, P. R. China.}
	\and
	Qia Li \thanks{School of Computer Science and Engineering, Sun Yat-sen University, Guangzhou 510275, P. R. China, Guangdong Province Key Laboratory of Computational Science  (liqia@mail.sysu.edu.cn). Questions, comments, or corrections to this document may be directed to that email address.}
}

\begin{document}

\maketitle
\begin{abstract}
	We consider in this paper a class of single-ratio fractional minimization problems, in which the numerator of the objective is the sum of a nonsmooth nonconvex function and a smooth nonconvex function while the denominator is a nonsmooth convex function. In this work, we first derive its first-order necessary optimality condition, by using the first-order operators of the three functions involved. Then we develop first-order algorithms, namely, the proximity-gradient-subgradient algorithm (PGSA), PGSA with monotone line search (PGSA\_ML) and PGSA with nonmonotone line search (PGSA\_NL). It is shown that any accumulation point of the sequence generated by them is a critical point of the problem under mild assumptions. Moreover, we establish global convergence of the sequence generated by PGSA or PGSA\_ML and analyze its convergence rate, by further assuming the local Lipschitz continuity of the nonsmooth function in the numerator, the smoothness of the denominator and the Kurdyka-{\L}ojasiewicz (KL) property of the objective. The proposed algorithms are applied to the sparse generalized eigenvalue problem associated with a pair of symmetric positive semidefinite matrices and the corresponding convergence results are obtained according to their general convergence theorems. We perform some preliminary numerical experiments to demonstrate the efficiency of the proposed algorithms.
\end{abstract}
\begin{keywords}
	fractional optimization, first-order algorithms, proximity algorithms, sparse generalized eigenvalue problem, KL property
\end{keywords}
\begin{AMS}
	90C26, 90C30, 65K05
\end{AMS}

\pagestyle{myheadings}
\thispagestyle{plain}
\section{Introduction}\label{section:introduction}
A fractional optimization problem is the problem which minimizes or maximizes an objective involving one or several ratios of functions. Fractional optimization problems arise from various applications in many fields, such as economics \cite{konno1989bond, pardalos1994use}, wireless communication \cite{Shen-Yu:2018IEEE, Zappone-Bjornson-Sanguinetti-Jorswieck:2017IEEE, Zappone-Sanguinetti-Debbah:2017IEEE}, artificial intelligence \cite{Baldacci-Lim-Traversi-Calvo:2018arXic, Hoyer2004Non} and so on.
Four categories of factional optimization problems, concerning minimizing a single ratio of two functions over a closed convex set, have been extensively studied in the literature. They are named according to the functions in the numerator and denominator: linear or quadratic fractional problems if both functions are linear or quadratic; convex-concave fractional problems if the numerator is convex and the denominator is concave; convex-convex fractional problems if both functions are convex. We refer the readers to \cite{Schaible:HRP1995, Schaible1983Fractional, stancu2012fractional}, for an overview on the single-ratio fractional optimization.

In this paper, we consider a class of single-ratio fractional minimization problems in the form of
\begin{equation}\label{problem:root}
\min\; \left\{\frac{f(x)+h(x)}{g(x)}:x \in \Omega\right\},
\end{equation}
where $f:\mathbb{R}^n\rightarrow\overline{\mathbb{R}}:=[-\infty,+\infty]$ is proper, lower semicontinuous, bounded below on $\mathbb{R}^n$ and continuous on its domain, $g:\mathbb{R}^n\rightarrow\mathbb{R}$ is convex, $h:\mathbb{R}^n\rightarrow\mathbb{R}$ is Lipschitz differentiable with a Lipschitz constant $L>0$, and the set $\Omega := \{x \in \mathbb{R}^n:g(x) \neq 0\}$ is nonempty.
Moreover, we assume that $f+h$ is non-negative on $\mathbb{R}^n$ and $g$ is positive on $\dom(f)\cap \Omega$.
Also, problem \eqref{problem:root} is assumed to have at least one optimal solution. It is obvious that both $f$ and $h$ are possibly nonconvex, while $f$ and $g$ can be nonsmooth. Problem \eqref{problem:root} does not belong to any of the four categories of fractional minimization problems aforementioned. This class of optimization problems subsumes a wide range of application models, e.g., the sparse generalized eigenvalue problem (SGEP)\cite{beck2013sparsity, Tan-Wang-Liu-Zhang:JRSS2018} and the $\ell_1/\ell_2$ sparse signal recovery problem \cite{Lou-Dong-Wang:2019SIAM}.

Now we turn to the algorithmic aspect of problem \eqref{problem:root}. To the best of our knowledge, this problem has seldom been studied in the literature and existing methods in general fractional optimization are not suitable for solving it. Global optimization methods, e.g., branch and bound algorithms \cite{BensonFractional:2006EJOR,KonnoMinimization}, play an important role in directly solving fractional optimization problems. However, the variable $x$ of problem \eqref{problem:root} is usually high dimensional in modern machine learning models. Thus, it is not practical to apply global optimization methods due to their expensive computational cost. For single fractional optimization problems, the variable transformation and parametric approach have been proposed to overcome the algorithmic difficulties caused by the ratio involved. In \cite{Charnes-Cooper:NRLQ1962}, Charnes and Cooper first suggested a variable transformation by which a linear fractional problem is reduced to a linear program. In fact, with the help of that variable transformation, any convex-concave fractional minimization problem can be equivalently reduced to a convex minimization problem. Since problem \eqref{problem:root} is not a convex-concave fractional minimization problem, through the variable transformation it remains nonconvex and in general difficult to solve. Hence, the variable transformation approach is not suitable for dealing with problem \eqref{problem:root}. Another widely used method for fractional optimization is the parametric approach, which takes good advantage of the relationship between a fractional problem and its associated parametric problem \cite{Dinkelbach-Werner:MS1967, Jagannathan:1966Management_Science}. Many efficient algorithms have been developed based on the parametric approach, see, for example, \cite{Dinkelbach-Werner:MS1967, Ibaraki:1983Mathematical_Programming, Pang:1980Operations_Research, Pang:1981Mathematical_Programming, Pardalos-Phillips:1991Journal_of_Global_Optimization}.
When they are applied to problem \eqref{problem:root}, most of these algorithms require to solve in each iteration a parametric subproblem in the form of
\begin{equation}\label{117E5}
\min~\{ f(x)+h(x)-cg(x):x\in\Omega \},
\end{equation}
where $c\in\mathbb{R}$ is determined by the previous iteration.
However, it is possibly not efficient enough since solving in each iteration a subproblem \eqref{117E5} would be numerically expansive.

In this work, we propose new iterative numerical algorithms for solving problem \eqref{problem:root}. In each iteration of the proposed algorithms, we mainly make use of the proximity operator of $f$, the gradient of $h$ and the subgradient of $g$ at the current iterate. When the above first-order operations are easy to compute, our algorithms perform efficiently. Our contributions are summarized below. 	
\begin{itemize}
	\item By Fr{\'e}chet subdifferentials of $f$, $g$ and the gradient of $h$, we derive a first-order necessary optimality condition for problem \eqref{problem:root} and thus introduce the definition of its critical points.	
	\item Based on the first-order optimality condition aforementioned, we develop for problem \eqref{problem:root} three first-order numerical algorithms, namely, proximity-gradient-subgradient algorithm (PGSA), PGSA with monotone line search (PGSA\_ML) and PGSA with nonmonotone line search (PGSA\_NL). Under mild assumptions on problem \eqref{problem:root}, we prove that any accumulation point of the sequence generated by any of the proposed algorithms is a critical point of problem \eqref{problem:root}. In addition, we show global convergence of the entire sequence generated by PGSA or PGSA\_ML, by further assuming that $f$ is locally Lipschitz in its domain, $g$ is differentiable with a locally Lipschitz continuous gradient and the objective in problem \eqref{problem:root} satisfies the Kurdyka-{\L}ojasiewicz property. The convergence rate of PGSA and PGSA\_ML are also estimated according to the Kurdyka-{\L}ojasiewicz property.
	\item We identify SGEP associated with a pair of symmetric positive semidefinite matrices as a special case of problem \eqref{problem:root} and apply the proposed algorithms to SGEP. We obtain the convergence results of the proposed algorithms for SGEP, by validating  all the conditions needed in their general convergence theorems. In particular, we prove that the sequence generated by PGSA or PGSA\_ML converges R-linearly by establishing that the KL exponent is $\frac{1}{2}$ at any critical point of SGEP.
\end{itemize}

The remaining part of this paper is organized as follows. In Section \ref{section:Notation_and_preliminaries}, we introduce notation and some necessary preliminaries. Section \ref{section:First-order condition} is devoted to a study of first-order necessary optimality conditions for problem \eqref{problem:root}. In Section \ref{section:PGSA}, we propose the PGSA and give its convergence analysis. In Section \ref{section5:PGSA_L}, we develop PGSA with line search (PGSA\_L), including PGSA\_ML and PGSA\_NL, and study their convergence property. We specify in Section \ref{section:sparse_generalized_eigenvalue} the proposed algorithms and convergence results obtained in Sections \ref{section:PGSA} and \ref{section5:PGSA_L} to the sparse generalized eigenvalue problem. In Section \ref{section:Numerical experiments}, some numerical results for SGEP and $\ell_1/\ell_2$ sparse signal recovery problem are presented to demonstrate the efficiency of the proposed algorithms. Finally, we conclude this paper in the last section.

\section{Notation and preliminaries}\label{section:Notation_and_preliminaries}
We start by our preferred notation. We denote by $\mathbb{N}$ the set of nonnegative integers. For a positive integer $n$, we let $\mathbb{N}_n := \{1,2,\cdots,n\}$ and $0_n$ be the $n$-dimensional zero vector. For $x\in\mathbb{R}$, let $[x]_+:=\max\{0,x\}$. By $\mathbb{S}^n_+$ (resp., $\mathbb{S}^n_{++}$) we denote the set of all $n\times n$ symmetric positive semidefinite (resp., definite) matrices. Given $H\in\mathbb{S}^n_{++}$, the weighted inner product of $x,y\in\mathbb{R}^n$ is defined by $\langle  x,y \rangle_H := \langle  x,Hy \rangle$ and the weighted $\ell_2$-norm of $x\in\mathbb{R}^n$ is defined by $\|x\|_H:=\sqrt{\innerP{x}{x}_H}$. For an $n\times n$ matrix $A$, we denote by $\|A\|_2$ the matrix 2-norm of $A$. For $\Lambda\subseteq\mathbb{N}_n$, let $|\Lambda|$ be the number of elements in $\Lambda$. We denote by $x_{\Lambda} \in \mathbb{R}^{|\Lambda|}$ the sub-vector of $x$ whose indices are restricted to $\Lambda$. We also denote by $A_{\Lambda}$ the $|\Lambda|\times|\Lambda|$ sub-matrix formed from picking the rows and columns of $A$ indexed by $\Lambda$. For a function $\varphi:\mathbb{R}^n\to\overline{\mathbb{R}}$ and $t\in\mathbb{R}$, let $\lev (\varphi,t) := \{x\in\mathbb{R}^n:\varphi(x)\leq t \}$.

For $x\in\mathbb{R}^n$, let $\supp(x)$ be the support of $x$, that is, $\supp(x):=\{i\in\mathbb{N}_n:x_i\neq 0\}$. Given $\delta>0$, we let $B(x,\delta):=\{z\in\mathbb{R}^n:\|z-x\|_2<\delta\}$ and $U(x,\delta):=\{z\in\mathbb{R}^n:|z_i-x_i|<\delta,\forall i\in\mathbb{N}_n\}$. For any closed set $S\subseteq \mathbb{R}^n$, the distance from $x\in\mathbb{R}^n$ to $S$ is defined by $\dist(x,S):=\inf\{\|x-z\|_2:z\in S\}$. The indicator function on $S$ is defined by
\begin{equation*}
\iota_S(x):=\begin{cases}
0,&\text{if }x\in S,\\
+\infty,&\text{otherwise}.
\end{cases}
\end{equation*}

In the remaining part of this section, we present some preliminaries on the Fr\'echet subdifferential and limiting-subdifferential \cite{Boris:2006Variational_analysis,Rockafellar2004Variational} as well as the Kurdyka-{\L}ojasiewicz (KL) property \cite{Attouch-bolt-redont-soubeyran:2010}. These concepts play a central role in our theoretical and algorithmic developments.


\subsection{Fr\'echet subdifferential and limiting-subdifferential}

Let $\varphi:\mathbb{R}^n\to\overline{\mathbb{R}}$ be a proper function. The domain of $\varphi$ is defined by $\dom (\varphi):=\{x\in\mathbb{R}^n:\varphi(x)< +\infty\}$. The Fr\'echet subdifferential of $\varphi$ at $x\in\dom(\varphi)$, denoted by $\widehat{\partial}\varphi(x)$, is defined by
\begin{equation*}
\widehat{\partial} \varphi (x):=\left\{ y\in\mathbb{R}^n:\mathop{\lim\inf}\limits_{\substack{z\to x\\z\neq x}}\;
\frac{\varphi(z)-\varphi(x)-\innerP{y}{z-x}}{\|z-x\|_2}\geq 0
 \right\}.
\end{equation*}
The set $\widehat{\partial}\varphi(x)$ is convex and closed. If $x\notin\dom(\varphi)$, we let $\widehat{\partial}\varphi(x) = \emptyset$. We say $\varphi$ is Fr\'echet subdifferentiable at $x\in\mathbb{R}^n$ when $\widehat{\partial}\varphi(x)\neq\emptyset$. Apart from the Fr\'echet subdifferential, we also need the notion of limiting-subdifferentials. The limiting-subdifferential or simply the subdifferential for short, of $\varphi$ at $x\in\dom(\varphi)$ is defined by
\begin{equation*} 
{\partial}\varphi(x) := \{ y\in\mathbb{R}^n:\exists x^k\to x,~\varphi(x^k)\to\varphi(x),~ y^k\in\widehat{\partial}\varphi(x^k)\to y \}.
\end{equation*}
It is straightforward that $\widehat{\partial}\varphi(x)\subseteq {\partial}\varphi(x)$ for all $x\in\mathbb{R}^n$. Moreover, if $\varphi$ is convex, then $\widehat{\partial}\varphi(x)$ and $\partial\varphi(x)$ reduce to the classical subdifferential in convex analysis, i.e.,
\begin{equation*}
\widehat{\partial} \varphi(x) = \partial\varphi(x) = \{y\in\mathbb{R}^n:\varphi(z)-\varphi(x)-\innerP{y}{z-x}\geq0,\forall z\in\mathbb{R}^n \}.
\end{equation*}

We next recall some simple and useful calculus results on $\widehat{\partial}$ and ${\partial}$. For any $\alpha>0$ and $x\in\mathbb{R}^n$, $\widehat{\partial}(\alpha\varphi)(x) = \alpha\widehat{\partial}\varphi(x)$ and ${\partial}(\alpha\varphi)(x) = \alpha{\partial}\varphi(x)$. Let $\varphi_1,~\varphi_2:\mathbb{R}^n\to\overline{\mathbb{R}}$ be proper and lower semicontinuous and $x\in\dom(\varphi_1 + \varphi_2)$. Then, $\widehat{\partial}\varphi_1(x) + \widehat{\partial}\varphi_2(x)\subseteq\widehat{\partial}(\varphi_1+\varphi_2)(x)$. If $\varphi_2$ is differentiable at $x$, then $\widehat{\partial}\varphi_2(x) = \{\triangledown\varphi_2(x) \}$ and $\widehat{\partial}(\varphi_1 + \varphi_2)(x) = \widehat{\partial}\varphi_1(x)+\triangledown\varphi_2(x)$. Furthermore, if $\varphi_2$ is continuously differentiable at $x$, then ${\partial}\varphi_2(x) = \{\triangledown\varphi_2(x) \}$ and ${\partial}(\varphi_1 + \varphi_2)(x) = {\partial}\varphi_1(x) + \triangledown\varphi_2(x)$.

We next present some results of the Fr\'echet subdifferential for the quotient of two functions. To this end, we first recall the calmness condition.
\begin{definition}[Calmness condition \cite{Rockafellar2004Variational}]
	The function $\varphi:\mathbb{R}^n\to\overline{\mathbb{R}}$ is said to satisfy the calmness condition at $x\in\dom(\varphi)$, if there exists $\kappa >0$ and a neighborhood $O$ of $x$, such that
	\begin{equation*}
	|\varphi(u)-\varphi(x)|\leq\kappa\|u-x\|_2
	\end{equation*}
	for all $u\in O$.
\end{definition}

The following proposition concerns the quotient rule of the Fr{\'e}chet subdifferential.

\begin{proposition}[Subdifferential calculus for quotient of two functions]\label{ppsition:2.2-1}
	Let $f_1:\mathbb{R}^n\to\overline{\mathbb{R}}$ be proper and $f_2:\mathbb{R}^n \to \mathbb{R}$. Define $\rho:\mathbb{R}^n\to\overline{\mathbb{R}}$ at $x\in\mathbb{R}^n$ as
	\begin{equation}\label{formula:quotient of two functions}
	\rho(x) :=
	\begin{cases}
	\frac{f_1(x)}{f_2(x)}, &\text{if } x\in\mathrm{dom}(f_1)\text{ and } f_2(x) \neq 0,\\
	+\infty, & \text{else.}
	\end{cases}
	\end{equation}
Let $x\in\mathrm{dom}(\rho)$ with $a_1 := f_1(x)$ and $a_2 := f_2(x) >0$. If $f_1$ is continuous at $x$ relative to $\dom(f_1)$ and $f_2$ satisfies the calmness condition at $x$, then
\begin{equation*}
\widehat{\partial}\rho(x) = \frac{\widehat{\partial}(a_2 f_1- a_1 f_2)(x)}{a_2^2}.
\end{equation*}
Furthermore, if $f_2$ is differentiable at $x$, then
\begin{equation*}
\widehat{\partial}\rho(x) = \frac{a_2 \widehat{\partial} f_1(x)- a_1 \triangledown f_2(x)}{a_2^2}.
\end{equation*}
\end{proposition}

The proof is given in the Appendix \ref{appendixA:proof_pposition2.2}.



\subsection{Kurdyka-{\L}ojasiewicz (KL) property}
\indent
\begin{definition}[KL property \cite{Attouch-bolt-redont-soubeyran:2010}]\label{Def:KL_property}
	A proper function $\varphi:\mathbb{R}^n\to\overline{\mathbb{R}}$ is said to satisfy the KL property at $\hat{x}\in\mathrm{dom}({\partial} \varphi)$ if there exist $\eta\in(0,+\infty]$, a neighborhood $O$ of $\hat{x}$ and a continuous concave function $\phi:[0,\eta) \to [0,+\infty]$, such that:
	\begin{enumerate}[\upshape(\romannumeral 1)]
		\item $\phi(0)=0$,
		\item $\phi$ is continuously differentiable on $(0,\eta)$ with $\phi'>0$,
		\item For any $x\in O \cap\{x\in\mathbb{R}^n:\varphi(\hat{x})<\varphi(x)<\varphi(\hat{x})+\eta \}$, there holds $\phi'(\varphi(x)-\varphi(\hat{x}))\mathrm{~dist}(0,{\partial}\varphi(x)) \geq 1$.
	\end{enumerate}
\end{definition}

A proper lower semicontinuous function $\varphi:\mathbb{R}^n\to\overline{\mathbb{R}}$ is called a KL function if $\varphi$ satisfies the KL property at all points in $\mathrm{dom}({\partial}\varphi)$. For connections between the KL property and the well-known error bound theory \cite{Luo-Pang:1994Mathematical_Programming, Luo-Tseng:1993Annals_of_Operations_Research, Pang:1997Mathematical_Programming}, we refer the interested readers to \cite{Bolte-Nguyen-Peypouquet:2017Mathematical_Programming, Li-Pong:2018Foundations_of_computational_mathematics}.
The notion of the KL property plays a crucial rule in analyzing the global sequential convergence. A framework for proving global sequential convergence using the KL property is provided in \cite{Attouch-Bolte:MP:2013}. We review this result in the next proposition.
\begin{proposition}\label{3.3ppsition}
	Let $\varphi:\mathbb{R}^n\to\overline{\mathbb{R}}$ be a proper lower semicontinuous function. Consider a sequence satisfying the following three conditions:
	\begin{enumerate}[\upshape(\romannumeral 1)]
		\item (Sufficient decrease condition.) There exists $a>0$ such that
		\begin{equation*}
		\varphi(x^{k+1}) + a \|x^{k+1} - x^k\|^2_2 \leq \varphi(x^k)
		\end{equation*}
		holds for any $k\in\mathbb{N}$;\label{3.3item1-1}
		\item (Relative error condition.) There exist $b>0$ and $\omega^{k+1} \in {\partial} \varphi(x^{k+1})$ such that
		\begin{equation*}
		\|\omega^{k+1}\|_2 \leq b\|x^{k+1} - x^k\|_2
		\end{equation*}
		holds for any $k\in\mathbb{N}$;\label{3.3item1-2}
		\item (Continuity condition.) There exist a subsequence $\{x^{k_j}:j\in\mathbb{N}\}$ and $x^{\star}$ such that
		\begin{equation*}
		x^{k_j} \to x^\star \text{ and } \varphi(x^{k_j}) \to \varphi(x^\star)\text{, as }j\to \infty.
		\end{equation*}\label{3.3item1-3}
	\end{enumerate}	
	If $\varphi$ satisfies the KL property at $x^{\star}$, then $\sum_{k=1}^{\infty}\|x^k-x^{k-1}\|_2<+\infty$, $\lim\limits_{k\to\infty} x^k=x^{\star}$ and $0\in{\partial}\varphi(x^{\star})$.
\end{proposition}

\section{First-order necessary optimality condition}\label{section:First-order condition}
In this section, we establish a first-order necessary optimality condition for local minimizers of problem \eqref{problem:root}. For convenience, we define $F:\mathbb{R}^n\to\overline{\mathbb{R}}$ at $x\in\mathbb{R}^n$ as
\begin{equation}\label{eq:1119biaohaoF}
F(x) :=
\begin{cases}
\frac{f(x)+h(x)}{g(x)}, & \text{if }x\in\Omega\cap\mathrm{dom}(f),\\
+\infty, & \text{else.}
\end{cases}
\end{equation}
Then, problem \eqref{problem:root} can be written as
\begin{equation*}
\min \{F(x):x\in\mathbb{R}^n \}.
\end{equation*}

From the generalized \textit{Fermat's rule} \cite[Theorem 10.1]{Rockafellar2004Variational}, we know that if $x^\star$ is a local minimizer of problem \eqref{problem:root} then $0\in\widehat{\partial} F(x^\star)$.
Since $g$ is not necessarily differentiable, in general $\widehat{\partial} F(x^{\star})$ can not be represented by Fr{\'e}chet subdifferentials of $f$ and $g$ and the gradient of $h$. Therefore, we have to derive the first-order optimality condition on a different manner.

Our idea is to take advantage of the parametric programing. With the help of the parametric problem, we obtain the first-order necessary optimality condition of local minimizers of $F$. To this end, we first characterize local and global minimizers of problem \eqref{problem:root} by those of its corresponding parametric problem. The result is presented in the next proposition and the proof is given in Appendix \ref{appendixA2:proof_pposition3.1}.
\begin{proposition}\label{2.2ppsition}
	Let $x^{\star} \in \mathrm{dom}(F)$ and $c_{\star} = F(x^\star)$. Then, $x^{\star}$ is a local (resp., global) minimizer of problem \eqref{problem:root} if and only if $x^{\star}$ is a local (resp., global) minimizer of the following problem:
	\begin{equation}\label{problem:root2}
	\min\;\{f(x)+h(x)-c_\star g(x):x\in \Omega \}.
	\end{equation}
\end{proposition}

We next present an important inequality, which plays a crucial role in deducing the first-order optimality condition.
\begin{lemma}\label{lemma:3.1-1}
	Let $x^\star \in \mathrm{dom}(F)$ be a local minimizer of problem \eqref{problem:root2} with $c_{\star} = F(x^\star)$. Then, there exists $\delta>0$ such that for any $x \in B(x^{\star},\delta)\cap\dom(F)$ and any $y^\star \in {\partial} g(x^\star)$, there holds
	\begin{equation*}
	f(x^\star) \leq f(x)+ \langle \triangledown h(x^\star)-c_\star y^\star, x-x^\star\rangle  + \frac{L}{2}\|x-x^\star\|_2^2.
	\end{equation*}
\end{lemma}
\begin{proof}
	Since $x^\star$ is a local minimizer of problem \eqref{problem:root2}, there exists $\delta>0$ such that for any $x \in B(x^{\star},\delta)\cap\dom(F)$, there holds
	\begin{equation}\label{formula:lemma3.0-1}
	f(x^\star) + h(x^\star) - c_\star g(x^\star) \leq f(x)+h(x)-c_\star g(x).
	\end{equation}
	Due to the Lipschitz continuity of $\triangledown h$, convexity of $g$ and $c_\star \geq 0$, it follows that, for any $x \in \mathbb{R}^n$ and $y^\star \in {\partial} g(x^\star)$,
	$h(x) \leq h(x^\star) + \langle \triangledown h(x^\star), x-x^\star \rangle + \frac{L}{2} \|x-x^\star\|^2_2$ and $c_\star g(x^\star) + \langle c_\star y^\star, x-x^\star \rangle \leq c_\star g(x).$
	By summing \eqref{formula:lemma3.0-1} and those two inequalities, we get this lemma.
\end{proof}

Now, we are ready to present the first-order necessary optimality condition for problem \eqref{problem:root}.
\begin{theorem}\label{theorem:3.0-1}
	Let $x^\star \in \mathrm{dom}(F)$ be a local minimizer of problem \eqref{problem:root} and $c_{\star} = F(x^\star)$, then $c_\star {\partial} g(x^\star) \subseteq \widehat{\partial} f(x^\star) + \triangledown h(x^\star)$.
\end{theorem}
\begin{proof}
	From Proposition \ref{2.2ppsition}, $x^\star$ is a local minimizer of problem \eqref{problem:root2}. By Lemma \ref{lemma:3.1-1}, we have that $x^\star$ is a local minimizer of the following problem, for all $y^\star \in {\partial} g(x^\star)$,
	\begin{equation*}
	\min\;\left\{f(x) + \langle \triangledown h(x^\star)-c_\star y^\star, x-x^\star \rangle + \frac{L}{2} \|x-x^\star\|^2_2:x\in\Omega \right\}.
	\end{equation*}
	Because $g$ is continuous on $\mathbb{R}^n$, $\Omega$ is an open subset of $\mathbb{R}^n$. Thus, $x^{\star}$ is an interior point of $\Omega$. Therefore, $0 \in \widehat{\partial} f(x^\star) + \triangledown h(x^\star) - c_\star y^\star$ for all $y^\star \in {\partial} g(x^\star)$. This implies that $c_\star {\partial} g(x^\star) \subseteq \widehat{\partial} f(x^\star) + \triangledown h(x^\star)$. We complete the proof.
\end{proof}

Inspired by the above theorem, we define a critical point of $F$ as follows.
\begin{definition}[Critical point of $F$]\label{Def:critical_point}
	Let $x^\star\in\mathrm{dom}(F)$ and $c_{\star} = F(x^\star)$. We say that $x^\star$ is a critical point of $F$ if
	\begin{equation*}
	0 \in \widehat{\partial} f(x^\star) + \triangledown h(x^\star) - c_\star {\partial} g(x^\star).
	\end{equation*}
\end{definition}

We remark that when $g$ is differentiable, by Proposition \ref{ppsition:2.2-1} we have for $x\in\dom(F)$,
\begin{align*}
\widehat{\partial} F(x) &= \frac{g(x)(\widehat{\partial} f(x) + \triangledown h(x)) - (f(x)+h(x))\triangledown g(x)}{g^2(x)}\\
&= \frac{1}{g(x)}(\widehat{\partial}f(x)+\triangledown h(x)-F(x) \triangledown g(x)).
\end{align*}
In this case, the statement that $x^\star$ is a critical point of $F$ (Definition \ref{Def:critical_point}) coincides with that $0\in\widehat{\partial} F(x^\star)$.

By Theorem \ref{theorem:3.0-1}, if $x^\star$ is a local minimizer of $F$, then $x^\star$ is a critical point of $F$. In the remaining part of this paper, we dedicate to developing iterative numerical algorithms to find critical points of $F$.

\section{The proximity-gradient-subgradient algorithm (PGSA) for solving problem \eqref{problem:root}}\label{section:PGSA}
This section is devoted to designing numerical algorithms for solving problem \eqref{problem:root}. We first propose an iterative scheme for solving problem \eqref{problem:root}, according to the first-order optimality condition. Then, we establish the convergence of objective function values and the subsequential convergence under a mild assumption. Finally, by making additional assumptions on $f$, $g$ and assuming the level boundedness and KL property of the objective, we prove the convergence of the whole sequence generated by the proposed algorithm.

From Theorem \ref{theorem:3.0-1}, a local minimizer of problem \eqref{problem:root} must be a critical point of $F$. Thus, our task becomes developing an algorithm with accumulation point being a critical point of $F$. To this end, we introduce the notion of proximity operators. For a proper and lower semicontinuous function $\varphi:\mathbb{R}^n \to \overline{\mathbb{R}}$, the proximity operator of $\varphi$ at $x\in\mathbb{R}^n$, denoted by $\text{prox}_\varphi(x)$, is defined by
\begin{equation*}
\text{prox}_\varphi(x) := \arg \min\; \{\varphi(y) + \frac{1}{2}\|y-x\|^2_2:y\in\mathbb{R}^n \}.
\end{equation*}
The operator $\text{prox}_\varphi$ is single-valued when $\varphi$ is convex and may be set-valued as $\varphi$ is nonconvex. With the help of the proximity operator, we derive a sufficient condition for a critical point of $F$ in the following proposition.
\begin{proposition}\label{3.0ppsition}
	If $x^\star\in\mathrm{dom}(F)$ satisfies
	\begin{equation}\label{formula:3.1-2x_star2}
	x^{\star} \in \mathrm{prox}_{\alpha f} (x^\star - \alpha \triangledown h(x^\star) + \alpha c_\star y^\star)
	\end{equation}
	for some $\alpha >0$, $y^\star\in{\partial} g(x^\star)$ and $c_\star = F(x^\star)$, then $x^\star$ is a critical point of $F$.
\end{proposition}
\begin{proof}
	By the proximity operator and the generalized \textit{Fermat's rule}, \eqref{3.0ppsition} leads to
	\begin{equation*}
	0\in\alpha\widehat{\partial} f(x^\star) + \alpha \triangledown h(x^\star) - \alpha c_\star y^\star,
	\end{equation*}
	which implies that $x^\star$ is a critical point of $F$.
\end{proof}

Inspired by Proposition \ref{3.0ppsition}, we propose the following first-order algorithm, which is stated in Algorithm \ref{alg:PGSA}. Since Algorithm \ref{alg:PGSA} involves in the proximity operator of $f$, the gradient of $h$ and the subgradient of $g$, we refer to it as the proximity-gradient-subgradient algorithm (PGSA).
\begin{algorithm}
	\caption{proximity-gradient-subgradient algorithm (PGSA) for solving \eqref{problem:root}} \label{alg:PGSA}
	\begin{tabular}{ll}
		Step 0. 				& Input $x^0\in\dom(F)$, $0<\underline{\alpha}\leq\alpha_k\leq\bar{\alpha}<1/L$, for $k\in\mathbb{N}$. Set $k\leftarrow 0$. \\		
		Step $1$. 			 	& Compute \\
								& \qquad $y^{k+1} \in {\partial} g(x^k)$, \\
								& \qquad $c_k = \DF{f(x^k)+h(x^k)}{g(x^k)}$, \\
								& \qquad $x^{k+1} \in \text{prox}_{\alpha_k f} (x^k - \alpha_k \triangledown h(x^k) + \alpha_k c_k y^{k+1})$.\\
	    Step $2$.				& Set $k\leftarrow k+1$ and go to Step 1.
	\end{tabular}
\end{algorithm}

In PGSA the step size $\alpha_k$ is required to be in $(0,1/L)$ to ensure $x^k \in \text{dom}(F)$ for all $k \in \mathbb{N}$. As a result the objective function value $c_k$ is well-defined. The detailed proof will be given in Lemma \ref{lemma:3.1}. Before starting the convergence analysis, we remark that PGSA differs from the classical parametric approach for problem \eqref{problem:root} combined with applying proximal subgradient (gradient) methods (e.g., see \cite{Gotoh-Takeda-Tono:2017Mathematical_Programming, Gotoh-Takeda-Tono:2017arXiv}) to the parametric subproblems involved. The parametric approach, which may date back to Dinkelbach's algorithm \cite{Dinkelbach-Werner:MS1967}, generates the new iterate of $k$-th iteration by solving a parametric subproblem
\begin{equation}\label{eq:reply_star}
x^{k+1} = \arg\min\{f(x)+h(x)-c_kg(x):x\in\Omega \},
\end{equation}
where $c_k$ is updated via $c_k:=\frac{f(x^k)+h(x^k)}{g(x^k)}$. In each iteration, one can apply proximal subgradient methods to subproblem \eqref{eq:reply_star}, which results in a type of algorithms combining the parametric approach and proximal subgradient methods for problem (1.1). However, these algorithms may be not efficient enough since solving subproblem \eqref{eq:reply_star} by proximal subgradient methods in each iteration can yield high computational cost.
On the other hand, the iterative procedure of PGSA can be equivalently reformulated as
\begin{align}
	x^{k+1} = \arg\min &\{ f(x)+h(x^k)-c_kg(x^k)\label{eq:reply_starstar}\\
	&+
	\langle \triangledown h(x^k)-c_ky^{k+1},x-x^k\rangle + \frac{\|x-x^k\|^2_2}{2\alpha_k} : x\in\mathbb{R}^n\},
	\notag
\end{align}
where $y^{k+1}\in\partial g(x^k)$ and $c_k =\frac{f(x^k)+h(x^k)}{g(x^k)}$. Comparing \eqref{eq:reply_star} and \eqref{eq:reply_starstar}, we see that instead of directly solving the parametric subproblem \eqref{eq:reply_star}, PGSA uses a quadratic approximation for $h(x)-c_kg(x)$ and then solves the resulting problem \eqref{eq:reply_starstar} in each iteration. It is worth noting that solving subproblem \eqref{eq:reply_starstar} is actually computing the proximity operator of $\alpha_k f$, which is usually much easier and more efficient than solving subproblem \eqref{eq:reply_star}.


\subsection{Convergence of objective function value}\label{ssection3.1:convergence_OBJfunc}
In this subsection, we prove that the sequence of the objective function values $\{F(x^k):k\in \mathbb{N}\}$ is decreasing and convergent. We first establish a lemma, which plays a crucial role in the convergence analysis.
\begin{lemma}\label{lemma:3.1}
	The sequence $\{x^k:k\in \mathbb{N}\}$ generated by PGSA falls into $\mathrm{dom}(F)$ and satisfies
	\begin{equation}\label{formula:3.1lemma1}
	f(x^{k+1}) + h(x^{k+1}) + \frac{1/\alpha_k - L}{2} \|x^{k+1}-x^k\|^2_2 \leq c_kg(x^{k+1}).
	\end{equation}
\end{lemma}
\begin{proof}
	We prove inequality \eqref{formula:3.1lemma1} and $x^k\in\mathrm{dom}(F)$ by induction. First, the initial point $x^0$ is in $\mathrm{dom}(F)$. Suppose $x^k\in\mathrm{dom}(F)$ for some $k\in\mathbb{N}$. From PGSA and the definition of proximity operators, we get
	\begin{align*}
	f(x^{k+1}) + \frac{1}{2\alpha_k} \|x^{k+1} - (x^k-\alpha_k \triangledown h(x^k) + \alpha_k c_k y^{k+1}) \|^2_2 \\
	\leq f(x^k) + \frac{1}{2\alpha_k} \|\alpha_k \triangledown h(x^k) - \alpha_k c_k y^{k+1} \| ^2_2,
	\end{align*}
	which implies that
	\begin{equation}\label{formula:3.1lemma2}
	f(x^{k+1}) + \frac{1}{2\alpha_k} \|x^{k+1} - x^{k}\|^2_2 + \langle x^{k+1} - x^k, \triangledown h(x^k) - c_k y^{k+1} \rangle \leq f(x^k).
	\end{equation}
	Since $\triangledown h$ is Lipschitz continuous with constant $L$, there holds
	\begin{equation}\label{formula:3.1lemma3}
	h(x^{k+1}) \leq h(x^k) + \langle \triangledown h(x^k), x^{k+1}-x^k \rangle + \frac{L}{2}\|x^{k+1} - x^k \|^2_2.
	\end{equation}
	Due to the convexity of $g$ and $c_k \geq 0$, it follows that
	\begin{equation}\label{formula:3.1lemma4}
	c_k g(x^k) + \langle c_k y^{k+1}, x^{k+1}-x^k \rangle \leq c_k g(x^{k+1}).
	\end{equation}
	By summing \eqref{formula:3.1lemma2}, \eqref{formula:3.1lemma3} and \eqref{formula:3.1lemma4}, we obtain \eqref{formula:3.1lemma1} from $c_k g(x^k) = f(x^k)+h(x^k)$.
	
	Assume that $x^{k+1}\notin\mathrm{dom}(F)$. We know $x^{k+1}\notin\Omega$ and $g(x^{k+1}) = 0$ due to $x^{k+1} \in\mathrm{dom}(f)$ and $\mathrm{dom}(F) = \Omega \cap \mathrm{dom}(f)$. By the fact $f+h\geq 0$ and $0<\alpha_k<1/L$, we deduce that $x^{k+1} = x^k$ from \eqref{formula:3.1lemma1}. This contradicts to $x^k\in\mathrm{dom}(F)$ and thus implies $x^{k+1}\in\mathrm{dom}(F)$. Therefore, we conclude $x^k\in\mathrm{dom}(F)$ for all $k\in\mathbb{N}$.
\end{proof}

With the help of Lemma \ref{lemma:3.1}, we get the main result of this subsection.
\begin{theorem}\label{3.1theorem1}
	Let $\{x^k:k\in \mathbb{N}\}$ be generated by PGSA. Then, the following statements hold:
	\begin{enumerate}[\upshape(\romannumeral 1)]
		\item $F(x^{k+1}) + \DF{1/\alpha_k-L}{2g(x^{k+1})}\|x^{k+1}-x^k\|^2_2 \leq F(x^k)\text{ for } k\in\mathbb{N};$
		\item $\lim\limits_{k\to\infty}c_k=\lim\limits_{k\to \infty}F(x^k)=c\text{ with }c \geq 0;$
		\item $\lim\limits_{k\to\infty}\DF{1/\alpha_k-L}{g(x^{k+1})}\|x^{k+1}-x^k\|^2_2=0$.
	\end{enumerate}
\end{theorem}

\begin{proof}
	From Lemma \ref{lemma:3.1}, $g(x^k) \neq 0$ for all $k\in \mathbb{N}$. Thus, \eqref{formula:3.1lemma1} in Lemma \ref{lemma:3.1} implies Item {\upshape(\romannumeral 1)} due to $g(x^{k+1})>0$. Item {\upshape(\romannumeral 2)} follows immediately by $F\geq 0$ and $0<\alpha_k<1/L$. Item {\upshape(\romannumeral 3)} is a direct consequence of Item {\upshape(\romannumeral 1)} and Item {\upshape(\romannumeral 2)}. We complete the proof.
\end{proof}


\subsection{Subsequential convergence}\label{ssection3.2:Subsequential_convergence}
In this subsection, we consider the subsequential convergence of PGSA. We begin with a mild assumption.
\begin{assumption}\label{assumption1: not 0 same time}
	Functions $f+h$ and $g$ do not attain 0 simultaneously.
\end{assumption}

With the help of Assumption \ref{assumption1: not 0 same time}, we can prove that $F$ is lower semicontinuous in the next proposition, which together with Theorem \ref{3.1theorem1} {\upshape(\romannumeral 1)} indicates that any accumulation point $x^{\star}$ of $\{x^k:k\in\mathbb{N} \}$ generated by PGSA is in $\dom(F)$, i.e., $g(x^{\star}) \neq 0$ and $x^{\star}\in\dom(f)$.

\begin{proposition}\label{lemma:1225star}
	Suppose Assumption \ref{assumption1: not 0 same time} holds. Then, $F$ is a lower semicontinuous function.
\end{proposition}
\begin{proof}
	If $x\in\Omega$, it holds that $0<g(x)=\lim\limits_{y\to x}g(y)$. Since $f$ is lower semicontinuous and $h$ is continuous, we immediately have $F(x)\leq \mathop{\lim\inf}\limits_{y\to x} F(y)$.
	If $x\notin\Omega$, we obtain that
	$F(x) = +\infty$ and $0 = g(x) = \lim\limits_{y\to x}g(y)$.
	Due to Assumption \ref{assumption1: not 0 same time}, $0<f(x) + h(x) \leq\mathop{\lim\inf}\limits_{y\to x}f(y) + h(y)$. Thus, $\mathop{\lim\inf}\limits_{y\to x}F(y) = +\infty$ from the fact that $g\geq 0$. Therefore, we have $F(x) = \mathop{\lim\inf}\limits_{y\to x}F(y)$. This completes the proof.
\end{proof}

To emphasize the importance of Assumption \ref{assumption1: not 0 same time}, we give an example below to illustrate that without Assumption \ref{assumption1: not 0 same time}, $F$ may not be lower semicontinuous and it is possible that $g$ vanishes at an accumulation point of $\{x^k:k\in\mathbb{N} \}$.
Consider the following one-dimensional fractional optimization problem:
\begin{equation*}
	\min\;\left\{\frac{\frac{1}{2}\sin^2 x}{|x|}:~x\neq 0 \right\},
\end{equation*}
where both $\frac{1}{2}\sin^2 x$ and $|x|$ attain zero at $x=0$, i.e., Assumption 1 is violated. Clearly, the corresponding $F$ is not lower semicontinuous at $x=0$ due to $F(0) = +\infty$ and $\lim\limits_{x\to 0}F(x) = 0$. Given an initial point $x^0\in(0,\pi/4)$ and a step size $\alpha_k \equiv \alpha\in(0,1)$. PGSA generates $\{x^k:k\in\mathbb{N} \}$ by
\begin{equation*}
	x^{k+1} = x^k-\frac{\alpha}{2}\sin(2x^k) + \frac{\frac{1}{2}\sin^2 x^k}{|x^k|}\alpha y^{k+1},
\end{equation*}
where $y^{k+1}\in{\partial}|\cdot|(x^k)$. Assume that $x^k\in(0,\pi/4)$. By applying the \textit{Lagrangian median theorem} for $\sin^2 x$ on $[0, x^k]$, we have
\begin{align*}
	x^{k+1} = x^k-\frac{\alpha}{2}\sin(2x^k) + \frac{\alpha\sin^2 x^k}{2x^k}
	=x^k -\frac{\alpha}{2}\sin(2x^k) + \frac{\alpha}{2}\sin(2\xi^k),
\end{align*}
for some $\xi^k\in(0,x^k)$. Therefore, invoking $x^0\in(0,\pi/4)$ and by induction on $k$, one can show that this $\{x^k:k\in\mathbb{N} \}$ is strictly decreasing and bounded below by zero, and thus is convergent. Finally, we can deduce that the limit point of $\{x^k:k\in\mathbb{N} \}$ is zero, which is infeasible in this fractional problem.

We are now ready to present the main result of this subsection.
\begin{theorem}\label{3.2theorem1}
	Suppose Assumption \ref{assumption1: not 0 same time} holds. Let $\{x^k:k\in\mathbb{N}\}$ be generated by PGSA. Then any accumulation point of $\{x^k:k\in\mathbb{N}\}$ is a critical point of $F$.
\end{theorem}
\begin{proof}
	Let $\{x^{k_j}:j\in\mathbb{N}\}$ be a subsequence such that $\lim\limits_{j\to\infty}x^{k_j} = x^\star$. By Theorem \ref{3.1theorem1} {\upshape(\romannumeral 1)} and Proposition \ref{lemma:1225star}, we deduce that $F(x^{\star})\leq\lim\limits_{j\to\infty}F(x^{k_j})\leq F(x^0)$, which indicates that $x^{\star}\in\dom(F)$, i.e., $g(x^{\star})\neq 0$ and $x^{\star}\in\dom(f)$. From Theorem \ref{3.1theorem1} {\upshape(\romannumeral 1)} and $\alpha_k\leq\bar{\alpha}$, we have
	\begin{equation*}
	F(x^{k_j}) + \DF{1/\bar{\alpha}-L}{2g(x^{k_j})}\|x^{k_j} - x^{k_j-1}\|^2_2 \leq F(x^{k_j-1}).
	\end{equation*}
	Using Item {\upshape(\romannumeral 2)} of Theorem \ref{3.1theorem1}, $\bar{\alpha}<1/L$ and the continuity of $g$ at $x^\star$, we conclude $\lim\limits_{j\to\infty} \|x^{k_j} - x^{k_j-1}\| = 0$ and $\lim\limits_{j\to\infty} x^{k_j-1} = x^\star$. Since $g$ is a real-valued convex function and $\{x^{k_j-1}:j\in\mathbb{N}\}$ is bounded, we know that $\{y^{k_j}:j\in\mathbb{N} \}$ is also bounded. Without loss of generality we may assume $\lim\limits_{j\to\infty} y^{k_j}$ and $\lim\limits_{j\to\infty} \alpha_{k_j-1}$ exist. In addition, $\lim\limits_{j\to\infty} y^{k_j} = y^\star$ belongs to ${\partial} g(x^\star)$ due to the closeness of ${\partial} g$. From the iteration of PGSA, we have
	\begin{equation}\label{formula:3.2theorem2}
	x^{k_j} \in \text{prox}_{\alpha_{k_j-1}f}(x^{k_j-1} - \alpha_{k_j-1}\triangledown h(x^{k_j-1}) + \alpha_{k_j-1}c_{k_j-1}y^{k_j}).
	\end{equation}
	As $\triangledown h$ and $f$ is continuous at $x^\star$, we obtain \eqref{formula:3.1-2x_star2} by passing to the limit in the above relation with $\alpha = \lim\limits_{j\to\infty}\alpha_{k_j-1}$. By Proposition \ref{3.0ppsition}, $x^\star$ is a critical point of $F$.	
\end{proof}


\subsection{Global sequential convergence}\label{ssection3.3:Global_convergence}
We investigate in this subsection the global convergence of the entire sequence $\{x^{k}:k\in \mathbb{N}\}$ generated by PGSA. We shall show $\{x^{k}:k\in \mathbb{N}\}$ converges to a critical point of $F$ under suitable assumptions. To this end, we need to introduce three assumptions as follows:
\begin{assumption}\label{assumption2: F-level-bounded}
	Function $F$ is level bounded.
\end{assumption}
\begin{assumption}\label{assumption3: f-Lcontinuous}
	Function $f$ is locally Lipschitz continuous on $\mathrm{dom}(f)$.
\end{assumption}
\begin{assumption}\label{assumption4: g-Ldifferentiable}
	Function $g$ is continuously differentiable on $\Omega$ with a locally Lipschitz continuous gradient.
\end{assumption}

Our analysis in this subsection mainly makes use of Proposition \ref{3.3ppsition} which is based on KL property. If $F$ is assumed to satisfy the KL property, from Proposition \ref{3.3ppsition} and Theorem \ref{3.2theorem1} we can establish the global convergence of PGSA by showing the boundedness of the sequence generated and Items {\upshape(\romannumeral 1)}-{\upshape(\romannumeral 2)} in Proposition \ref{3.3ppsition}. The boundness of $\{x^{k}:k\in \mathbb{N}\}$ is a direct consequence of Theorem \ref{3.1theorem1} {\upshape(\romannumeral 1)} and Assumption \ref{assumption2: F-level-bounded}. Other results needed will be proved in the following two lemmas.

\begin{lemma}\label{lemma:3.2}
	Suppose that Assumptions \ref{assumption1: not 0 same time} and \ref{assumption2: F-level-bounded} hold. Let $\{x^k:k\in\mathbb{N}  \}$ be generated by PGSA. Then the following statements hold:
	\begin{enumerate}[{\upshape(\romannumeral 1)}]
		\item $\{x^k:k\in\mathbb{N}  \}$ is bounded;
		\item $F(x^{k+1})+\frac{a}{2}\|x^{k+1}-x^k\|^2_2\leq F(x^k)$ for $k\in\mathbb{N}$, where $a := (1/\bar{\alpha}-L)/M>0$ with $M:=\sup\{g(x):x\in\lev(F,c_0)  \}$.
	\end{enumerate}
\end{lemma}
\begin{proof}
	By Theorem \ref{3.1theorem1} {\upshape(\romannumeral 1)}, we have for all $k\in\mathbb{N}$, $x^k\in\lev(F,c_0)$. Then the boundedness of $\{x^k:k\in\mathbb{N} \}$ follows immediately from Assumption \ref{assumption2: F-level-bounded}. In view of Proposition \ref{lemma:1225star}, Assumption \ref{assumption1: not 0 same time} ensures the lower semicontinuity of $F$. Hence, the set $\lev(F,c_0)$ is closed and bounded. Since $g$ is continuous, we know $M$ is finite. This together with Theorem \ref{3.1theorem1} {\upshape(\romannumeral 1)} and $\alpha_k < \bar{\alpha}$ yields Item {\upshape(\romannumeral 2)}.
\end{proof}

\begin{lemma}\label{lemma:3.3}
	Let $\{x^k:k\in\mathbb{N}\}$ be generated by PGSA. Suppose Assumptions \ref{assumption1: not 0 same time}-\ref{assumption4: g-Ldifferentiable} hold. Then there exist $b>0$ and $\omega^{k+1}\in{\partial} F(x^{k+1})$ such that
	\begin{equation*}
	\|\omega^{k+1}\|_2\leq b\|x^{k+1}-x^k\|_2
	\end{equation*}
	for all $k\in\mathbb{N}$.	
\end{lemma}

\begin{proof}
	By Lemma \ref{lemma:3.1} and Theorem \ref{3.2theorem1}, we know $x^k \in \Omega$ for any $k \in \mathbb{N}$ and any accumulation point $x^\star$ of $\{x^k:k\in\mathbb{N}\}$ satisfies $g(x^\star) >0$. Thus, there exists $t>0$ such that $g(x^k) \geq t$, since $\{x^k:k\in\mathbb{N}\}$ is bounded and $g$ is continuous on $\Omega$. Let $S\subseteq \mathbb{R}^n$ be a bounded closed subset satisfying $\{x^k:k\in\mathbb{N}\}\subseteq S\subseteq \mathrm{dom}(F)$. Then it is easy to check that $\triangledown g$ and $F$ are globally Lipschitz continuous on $S$. We denote the Lipschitz constant of $\triangledown g$ and $F$ by $\hat{L}$ and $\widetilde{L}$ respectively.
	
	From the iteration of PGSA and the differentiability of $g$, we obtain that
	\begin{equation*}
	x^k-x^{k+1}-\alpha_k \triangledown h(x^k) + \alpha_k c_k \triangledown g(x^k) \in \alpha_k \widehat{\partial} f(x^{k+1}),
	\end{equation*}
	which implies that
	\begin{equation}\label{formula:3.3lemma1}
	\frac{1}{\alpha_k g(x^{k+1})}(x^k-x^{k+1}) - \frac{\triangledown h(x^k)}{g(x^{k+1})} + \frac{c_k}{g(x^{k+1})} \triangledown g(x^k) \in \frac{\widehat{\partial} f(x^{k+1})}{g(x^{k+1})}.
	\end{equation}
    From Assumptions \ref{assumption3: f-Lcontinuous}-\ref{assumption4: g-Ldifferentiable} and Proposition \ref{ppsition:2.2-1}, we have on $\mathrm{dom}(\widehat{\partial} F)$
	\begin{equation*}
	\widehat{\partial} F = \frac{g(\widehat{\partial} f + \triangledown h) - (f+h)\triangledown g}{g^2}.
	\end{equation*}
	The above relation and \eqref{formula:3.3lemma1} suggest that $\omega^{k+1} \in \widehat{\partial} F(x^{k+1})$ with
	\begin{align*}
	\omega^{k+1} := &\frac{1}{\alpha_k g(x^{k+1})}(x^k-x^{k+1}) - \frac{\triangledown h(x^k)}{g(x^{k+1})} + \frac{\triangledown h(x^{k+1})}{g(x^{k+1})} \\
	&+ \frac{c_k}{g(x^{k+1})} \triangledown g(x^k) - \frac{c_{k+1}}{g(x^{k+1})} \triangledown g(x^{k+1}).
	\end{align*}
	By a direct computation, it follows that
	\begin{equation}\label{formula:3.3lemma2}
	\|\omega^{k+1}\|_2 \leq \left(\frac{1}{\alpha_k t} +\frac{L}{t} + \frac{c_k\hat{L}}{t} + \frac{\|\triangledown g(x^{k+1})\|_2\widetilde{L}}{t} \right) \|x^{k+1}-x^k\|_2.
	\end{equation}
	From Theorem \ref{3.1theorem1}, we see that $c_k \leq c_0$ for $k \in \mathbb{N}$. Since $\{x^k:k\in\mathbb{N}\}$ is bounded and $\triangledown g$ is continuous on $\Omega$, there exists $\beta > 0$ such that $\|\triangledown g(x^{k+1})\|_2 \leq \beta$ for all $k\in\mathbb{N}$. Due to $\alpha_k \geq \underline{\alpha} > 0$, we obtain finally from \eqref{formula:3.3lemma2} that $\|\omega^{k+1}\|_2 \leq b \|x^{k+1}-x^k\|_2$ for all $k\in\mathbb{N}$, where $b:=(1/\underline{\alpha}+L+c_0\hat{L}+\beta\widetilde{L})/a$. We complete the proof due to $\widehat{\partial} F(x^{k+1})\subseteq {\partial} F(x^{k+1})$.
\end{proof}

Now we are ready to present the main result of this subsection.
\begin{theorem}\label{theorem:4.3.1-1}	
	Suppose that Assumptions \ref{assumption1: not 0 same time}-\ref{assumption4: g-Ldifferentiable} hold and $F$ satisfies the KL property at any point in $\dom(F)$. Let $\{x^k:k\in\mathbb{N}\}$ be generated by PGSA. Then $\sum_{k=1}^{\infty}\|x^k-x^{k-1}\|_2<+\infty$ and $\{x^k:k\in\mathbb{N}\}$ converges to a critical point of $F$.
\end{theorem}
\begin{proof}
	From Theorem \ref{3.2theorem1}, it suffices to prove that $\sum_{k=1}^{\infty}\|x^k-x^{k-1}\|_2 < +\infty$ and $\{x^k:k\in\mathbb{N}\}$ is convergent. According to Proposition \ref{3.3ppsition}, we obtain this theorem from Lemmas \ref{lemma:3.2}-\ref{lemma:3.3} and Proposition \ref{lemma:1225star} immediately.
\end{proof}


\subsection{Convergence rate}\label{ssection3.4:Convergence_rate_analysis}
Finally, we consider the convergence rate of PGSA. To this end, we further assume $F$ is a KL function with the corresponding $\phi$ (see Definition \ref{Def:KL_property}) taking the form $\phi(s) = ds^{1-\theta}$ for some $d>0$ and $\theta\in[0,1)$. Then under the assumption of Theorem \ref{theorem:4.3.1-1}, we can estimate the convergence rate of PGSA, following a similar line of arguments to other convergence rate analysis based on the KL property; see, for example, \cite{Attouch-Bolte:MP2009, Wen-Chen-Pong:2018COA,Xu-Yin:SIAMIS:13}.
\begin{theorem}\label{theorem:4.3.1-2}
	Suppose that Assumptions \ref{assumption1: not 0 same time}-\ref{assumption4: g-Ldifferentiable} hold. Let $\{x^k:k\in\mathbb{N}\}$ be generated by PGSA and suppose that $\{x^k:k\in\mathbb{N}\}$ converges to $x^{\star}$. Assume further that $F$ satisfies the KL property at $x^{\star}$ with $\phi(s) = ds^{1-\theta}$ for some $d>0$ and $\theta\in[0,1)$, then the following statements hold:
	\begin{enumerate}[{\upshape(\romannumeral 1)}]
		\item If $\theta = 0$, $\{x^k:k\in\mathbb{N} \}$ converges to $x^\star$ finitely;
		\item If $\theta \in (0,1/2]$, $\|x^k-x^\star\|_2\leq c_1 \tau^k$, $\forall k\geq K_1$ for some $K_1>0$, $c_1>0$, $\tau\in(0,1)$;
		\item If $\theta \in (1/2,1)$, $\|x^k-x^\star\|_2\leq c_2 k^{-(1-\theta)/(2\theta -1)}$, $\forall k\geq K_2$, for some $k\geq K_2$, $c_2>0$.
	\end{enumerate}	
\end{theorem}
Here we omit the proof for Theorem \ref{theorem:4.3.1-2}, since it can be performed very similarly to those for other optimization algorithms (see, for example, the proof of \cite[Theorem 2]{Attouch-Bolte:MP2009}).  We remark that as is pointed out in \cite{Attouch-bolt-redont-soubeyran:2010}, all proper semialgebraic functions satisfy the KL property with $\phi(s) = ds^{1-\theta}$ for some $d>0$ and $\theta\in[0,1)$. Consequently, both Theorems \ref{theorem:4.3.1-1} and \ref{theorem:4.3.1-2} are applicable when $F$ is a semialgebraic function. Indeed, the objective functions are semialgebraic in a wide range of sparse optimization problems, including the sparse generalized eigenvalue problem \eqref{117E4} which will be studied in detail in Section \ref{section:sparse_generalized_eigenvalue}.

To close this section, we point out that when $f$ is convex, the following inequality instead of \eqref{formula:3.1lemma2} will be obtained in Lemma \ref{lemma:3.1}:
\begin{equation*}
f(x^{k+1}) + h(x^{k+1}) + \left( \frac{1}{\alpha_k}-\frac{L}{2} \right) \|x^{k+1}-x^k\|^2_2\leq c_kg(x^{k+1}).
\end{equation*}
As a consequence, one can easily verify that all the convergence results established in subsections \ref{ssection3.2:Subsequential_convergence}-\ref{ssection3.4:Convergence_rate_analysis} still hold for PGSA with $0<\underline{\alpha}\leq\alpha_k\leq\bar{\alpha}<2/L$ in the case where $f$ is convex.

\section{PGSA with line search}\label{section5:PGSA_L}
In this section, we incorporate a line search scheme for adaptively choosing $\alpha_k$ into PGSA. In PGSA, the step size $\alpha_k$ should be less than $1/L$ for all $k\in\mathbb{N}$ to ensure the convergence. However, this step size may be too small in the case of large $L$ and thus leads to slow convergence of PGSA. To speed up the convergence, we take advantage of the line search technique in \cite{ Lu-Zhou:SIAM2019,Gotoh-Takeda-Tono:2017arXiv, Wright-Nowak-Figueiredo:IEEE-TSP2009} to enlarge the step size and meanwhile guarantee the convergence of the algorithm. The PGSA with line search is summarized in Algorithm \ref{alg:PGSAL} (PGSA\_L).
\begin{algorithm}
	\caption{PGSA with line search (PGSA\_L) for problem \eqref{problem:root}} \label{alg:PGSAL}
	\begin{tabular}{ll}
		Step $0$.  &Input $x^0\in\dom(F)$, $a>0$, $0<\underline{\alpha}< \bar{\alpha}$, $0<\eta<1$, and an integer $N\geq 0$. \\
				   &Set $k\leftarrow 0$.\\
		Step $1$.  &$y^{k+1}\in{\partial} g(x^k)$,\\
		&$c_k=\frac{f(x^k)+h(x^k)}{g(x^k)}$,\\
		&Choose $\alpha_{k,0}\in[\underline{\alpha},\bar{\alpha}]  $.\\
		Step $2$. &For $m=0,1,\dots,$ do\\
		&\qquad $\alpha_k = \alpha_{k,0}\eta^m$,\\
		&\qquad $\tilde{x}^{k+1}\in\mathrm{prox}_{\alpha_k f}(x^k-\alpha_k\triangledown h(x^k) + \alpha_k c_k y^{k+1})$,\\
		&\qquad If $\tilde{x}^{k+1}$ satisfies $\tilde{x}^{k+1}\in\dom(F)$ and
	\end{tabular}
	\begin{equation}\label{118E2}
	F(\tilde{x}^{k+1})\leq\max\limits_{[k-N]_+\leq j\leq k}c_j-\frac{a}{2}\|\tilde{x}^{k+1}-x^k\|^2_2,
	\end{equation}
	\begin{tabular}{ll}
		&\qquad set $x^{k+1}=\tilde{x}^{k+1}$ and go to Step 3.\\
		Step $3$.  &$k\leftarrow k+1$ and go to Step 1.
	\end{tabular}
\end{algorithm}

From inequality \eqref{118E2}, $\{F(x^k):k\in\mathbb{N}\}$ is monotone when $N=0$, while it is generally nonmonotone when $N>0$. For convenience of presentation, we call the algorithm PGSA with monotone line search (PGSA\_ML) if $N=0$ and PGSA with nonmonotone line search (PGSA\_NL) if $N>0$. Let $\Delta x:=x^k-x^{k-1}$, $\Delta h := \triangledown h(x^k) - \triangledown h(x^{k-1})$. Motivated from \cite{Barzilai-Borwein:IJNA1998, Lu-Zhou:SIAM2019, Wright-Nowak-Figueiredo:IEEE-TSP2009}, we adopt a very popular choice of $\alpha_{k,0}$ in the following formula
\begin{equation}\label{120E1}
\alpha_{k,0} =
\begin{cases}
\max\left\{ \underline{\alpha},\min\{\bar{\alpha},\frac{\|\Delta x\|^2_2}{|\innerP{\Delta x}{\Delta h}|} \} \right\}, & \text{if }\innerP{\Delta x}{\Delta h}\neq 0,\\
\bar{\alpha},& \text{else.}
\end{cases}
\end{equation}
This initial step size can be viewed as an adaptive approximation of $1/L$ via some local curvature information of $h$.

Next, we study the convergence property of PGSA\_L. To this end, we define $\tau:\mathbb{N}\to\mathbb{N}$ at $k\in\mathbb{N}$ as $\tau(k) := \max\{i:i\in\arg\max\{F(x^j):[k-N]_+\leq j\leq k \} \}$. The following lemma tells that PGSA\_L is well defined and the sequence generated by PGSA\_L is bounded under Assumption 2.
\begin{lemma}\label{131lemma5.4}
	Suppose that Assumption 2 holds and let $M:=\sup \{g(x):x\in \lev(F,c_0)  \}$. Then, the following statements hold:
	\begin{enumerate}[{\upshape(\romannumeral 1)}]
		\item Step 2 of PGSA\_L terminates at some $\alpha_k\geq\tilde{\alpha}$ in most $T$ iterations, where $\tilde{\alpha} := \eta/(aM+L)$, $T:=\lceil \frac{-\log(\bar{\alpha}(aM+L))}{\log \eta} +1 \rceil$;
		\item $x^k\in\lev(F,c_0)$ for all $k\in\mathbb{N}$ and thus $\{x^k:k\in\mathbb{N}\}$ is bounded;
		\item $\{F(x^{\tau(k)}):k\in\mathbb{N} \}$ is nonincreasing.
	\end{enumerate}
\end{lemma}
\begin{proof}
	Assumption 2 ensures the boundedness of $\lev(F,c_0)$. Thus, we know $M$ is finite thanks to the continuity of $g$. In view of the updating rule for $\alpha_k$ in Step 2 and $\alpha_{k,0} \leq \bar{\alpha}$, after $T$ iterations, we have $\alpha_k \leq 1/(aM+L) = \tilde{\alpha}/\eta$ for any $k\in\mathbb{N}$.
	
	We proceed by induction on $k$. It is obvious that $x^0\in\lev(F,c_0)$. Now, assume that for $j=0,1,...,k$, $x^j$ has already been generated and $x^j\in\lev(F,c_0)$. In order to prove Item {\upshape(\romannumeral 1)}, it suffices to show that if $\alpha_k\leq\tilde{\alpha}/\eta$, then $\tilde{x}^{k+1}\in\dom(F)$ and the following inequality holds
	\begin{equation}\label{131Estar}
	F(\tilde{x}^{k+1})\leq c_k - \frac{a}{2}\|\tilde{x}^{k+1}-x^k\|^2_2.
	\end{equation}
	By Theorem \ref{3.1theorem1} and $\alpha_k\leq 1/(aM+L)<1/L$, we have $\tilde{x}^{k+1}\in\dom(F)$ and	
	\begin{equation}\label{131Etriangle}
	F(\tilde{x}^{k+1})\leq c_k-\frac{1/\alpha_k - L}{2g(\tilde{x}^{k+1})}\|\tilde{x}^{k+1}-x^k\|^2_2
	\leq c_k - \frac{aM}{2g(\tilde{x}^{k+1})}\|\tilde{x}^{k+1}-x^k\|^2_2,
	\end{equation}
	which indicates that $F(\tilde{x}^{k+1})\leq c_k\leq c_0$ and thus $\tilde{x}^{k+1}\in \lev(F,c_0)$. Invoking $g(\tilde{x}^{k+1})\leq M$, we obtain inequality \eqref{131Estar} from \eqref{131Etriangle}.
	
	We next prove $x^{k+1}\in \lev(F,c_0)$ and $F(x^{\tau(j+1)})\leq F(x^{\tau(j)})$ for $j\leq k$. By \eqref{118E2}, we have $F(x^{j+1})\leq F(x^{\tau(j)})$ for $j\leq k$. Thus, for $j\leq k$,
	\begin{align*}
	F(x^{\tau(j+1)}) &= \max\limits_{[j+1-N]_+\leq i\leq j+1} F(x^i)\\
	&= \max\left\{F(x^{j+1}),\max\limits_{[j+1-N]_+\leq i\leq j} F(x^i)  \right\}\\
	&\leq \max\{F(x^{j+1}),F(x^{\tau(j)})\}\\
	&= F(x^{\tau(j)}).
	\end{align*}
	This yields that $F(x^{k+1})\leq F(x^{\tau(k)}) \leq F(x^{\tau(0)}) = c_0$. We complete the proof immediately.
\end{proof}

With the help of Lemma \ref{131lemma5.4}, we establish the subsequential convergence results of PGSA\_L in the next theorem.
\begin{theorem}\label{theorem:118triangle}
	Suppose that Assumptions 1 and 2 hold. Let $\{x^k:k\in\mathbb{N}\}$ be generated by PGSA\_L. Then any accumulation point of $\{x^k:k\in\mathbb{N}\}$ is a critical point of $F$.
\end{theorem}
\begin{proof}
	Let $x^{\star}$ be an accumulation point of $\{x^k:k\in\mathbb{N}\}$. According to the proof of Theorem \ref{3.2theorem1}, it suffices to show $\{F(x^k):k\in\mathbb{N}\}$ converges and $\lim\limits_{k\to\infty}\|x^{k+1}-x^k\|_2=0$. By Lemma \ref{131lemma5.4},  $\{F(x^{\tau(k)}):k\in\mathbb{N}\}$ is decreasing and $F\geq 0$. Hence, we have that $\lim\limits_{k\to\infty}F(x^{\tau (k)}) = \xi$ for some $\xi\geq 0$. Since $f$ is continuous on $\dom(f)$ and $\lev(F,c_0)$ is closed and bounded, we deduce that $F$ is uniformly continuous on $\lev(F,c_0)$. Noting that $\{x^k:k\in\mathbb{N}\}\subseteq \lev(F,c_0)$ and proceeding as in the proof of \cite[Lemma 4]{Wright-Nowak-Figueiredo:IEEE-TSP2009} starting from \cite[Equation (34)]{Wright-Nowak-Figueiredo:IEEE-TSP2009}, one can prove that $\lim\limits_{k\to\infty}F(x^k)=\xi$ and $\lim\limits_{k\to\infty}\|x^{k+1}-x^k\|_2=0$. We complete the proof.
\end{proof}

Under Assumptions 1-4 and assuming $F$ satisfies the KL property, we can prove the global convergence of the entire sequence generated by PGSA\_ML.

\begin{theorem}\label{theorem:118star}
	Suppose that Assumptions 1-4 hold and $F$ satisfies the KL property at any point in $\dom(F)$. Let $\{x^k:k\in\mathbb{N}\}$ be generated by PGSA\_ML. Then $\sum_{k=1}^{\infty}\|x^k-x^{k-1}\|_2<+\infty$ and $\{x^k:k\in\mathbb{N}\}$ converges to a critical point of $F$.
\end{theorem}
\begin{proof}
	From Theorem \ref{theorem:118triangle}, it suffices to prove that $\sum_{k=1}^{\infty}\|x^k-x^{k-1}\|_2<+\infty$ and $\{x^k:k\in\mathbb{N}\}$ is convergent. According to Proposition \ref{3.3ppsition}, we need to verify Items {\upshape(\romannumeral 1)}-{\upshape(\romannumeral 3)} of the proposition, the boundedness of $\{x^k:k\in\mathbb{N}\}$ and that $F$ is lower semicontinuous.
	
	First, the boundedness of $\{x^k:k\in\mathbb{N}\}$ and lower semicontinuity of $F$ follow from Lemma \ref{131lemma5.4} and Proposition \ref{lemma:1225star}, respectively. Items {\upshape(\romannumeral 1)} and {\upshape(\romannumeral 3)} of Proposition \ref{3.3ppsition} are direct consequence of Lemma \ref{131lemma5.4} and Theorem \ref{theorem:118triangle}. Proposition \ref{3.3ppsition} {\upshape(\romannumeral 2)} can be obtained by a proof similar to that of Lemma \ref{lemma:3.3}. Therefore, we complete the proof.
\end{proof}

The convergence rate analysis of PGSA\_ML is almost the same as that of PGSA in Theorem \ref{theorem:4.3.1-2}. Here, we omit the details and present the corresponding results in the next theorem.
\begin{theorem}\label{theorem:202 5.4}
	Suppose that Assumptions 1-4 hold. Let $\{x^k:k\in\mathbb{N}\}$ be generated by PGSA\_ML and suppose that $\{x^k:k\in\mathbb{N}\}$ converges to $x^{\star}$. Assume further that $F$ satisfies the KL property at $x^{\star}$ with $\phi(s) = ds^{1-\theta}$ for some $d>0$ and $\theta\in[0,1)$, then Items {\upshape(\romannumeral 1)}-{\upshape(\romannumeral 3)} of Theorem \ref{theorem:4.3.1-2} hold.
\end{theorem}

\section{Applications to sparse generalized eigenvalue problem}\label{section:sparse_generalized_eigenvalue}
In this section, we identify SGEP associated with a pair of symmetric positive semidefinite matrices as a special case of problem \eqref{problem:root} and apply our proposed algorithms. Then we establish the global sequential (resp., subsequential) convergence of the sequence generated by PGSA and PGSA\_ML (resp., PGSA\_NL)  for SGEP. In addition, we prove that the sequence generated by PGSA or PGSA\_ML converges R-linearly by establishing that the KL exponent is $\frac{1}{2}$ at any critical point of SGEP.

Assume that both $A$, $B$ are $n\times n$ symmetric positive semidefinite matrices and any $r\times r$ principal sub-matrix of $B$ is positive definite for some integer $r\in [1,n]$. If there exist $\lambda^{\star}\in\mathbb{R}$ and $x^{\star}\in\mathbb{R}^n$, such that
$
Ax^{\star} = \lambda^{\star}Bx^{\star},
$
then $x^\star$ is called the generalized eigenvector with respect to the generalized eigenvalue $\lambda^{\star}$ of the matrix pair $(A,B)$. Obviously, the leading generalized eigenvector with respect to the largest generalized eigenvalue can be obtained by solving the following optimization problem
\begin{equation}\label{117E2}
\max\; \left\{ \frac{x^TAx}{x^TBx}:\|x\|_2=1,~x^TBx\neq 0, ~x\in\mathbb{R}^n \right\}.
\end{equation}
In the context of sparse modeling, it is natural to incorporate the sparsity constraint into problem \eqref{117E2}. This leads to the SGEP:
\begin{equation}\label{117E3}
\max\; \left\{ \frac{x^TAx}{x^TBx}:\|x\|_2=1,~\|x\|_0\leq r, ~x^TBx\neq 0, ~x\in\mathbb{R}^n \right\},
\end{equation}
where the $\ell_0$ function $\|\cdot\|_0$ counts the number of nonzero components in a vector. The SGEP covers several statical learning models, such as the sparse principle component analysis \cite{d2008optimal, zou2006sparse}, sparse fisher discriminant analysis \cite{Clemmensen-Hastie-Witten:Technometrics2011, Mai-Zou-Yuan:Biometrika2012}, sparse sliced inverse regression \cite{Chen-Zou-Cook:The_Annals_of_Statistics2010, Li-Nachtsheim:Technometrics2006} and so on. One can easily check that the optimal solution set of SGEP is completely the same as that of the following minimization problem
\begin{equation}\label{117E4}
\min\; \left\{ \frac{x^TBx}{x^TAx}:\|x\|_2=1,~\|x\|_0\leq r,~x^TAx\neq 0,~x\in\mathbb{R}^n \right\}.
\end{equation}
Thus, problem \eqref{117E4} is another formulation of SGEP. We also notice that problem \eqref{117E4} is not a classical quadratic fractional problem due to its nonconvex constraints. In fact, problem \eqref{117E4} is a special case of the general optimization problem \eqref{problem:root} with $f$ being the indicator function on the set $\{ x\in\mathbb{R}^n: \|x\|_0\leq r,~\|x\|_2=1 \}$, $g(x) = x^TAx$, $h(x) = x^TBx$ for $x\in\mathbb{R}^n$. Therefore, the proposed PGSA and PGSA\_L can be directly applied to problem \eqref{117E4}. For convenience of presentation, we denote the constraint set $\{x\in\mathbb{R}^n:\|x\|_0\leq r,~\|x\|_2=1  \}$ in problem \eqref{117E4} by $C$ and define $G:\mathbb{R}^n\to\overline{\mathbb{R}}$ at $x\in\mathbb{R}^n$ as
\begin{equation*}
G(x):=
\begin{cases}
\DF{x^TBx}{x^TAx}, &\text{if~}x\in C\text{~and~} x^TAx\neq 0,\\[8pt]
+\infty,&\text{else.}
\end{cases}
\end{equation*}

\subsection{Critical points of problem \eqref{117E4}}
In this subsection, we have a closer look at the critical points of problem \eqref{117E4}. We begin with the following lemma concerning the Fr\'{e}chet subdifferential of the indicator function $\iota_C$.
\begin{lemma}\label{lemma:7.0-1}
	Let $x\in C$ and $\Lambda$ be the support of $x$, then the following statements hold:
	\begin{enumerate}[\upshape(\romannumeral 1)]
		\item $\widehat{\partial} \iota_C(x) =
		\begin{cases}
		\{v\in\mathbb{R}^n:v=tx,t\in\mathbb{R} \}, & \text{if~~} \|x\|_0<r,\\
		\{v\in\mathbb{R}^n:v_\Lambda = tx_\Lambda, t\in\mathbb{R} \}, &\text{else.}
		\end{cases}$
		\item For any $v\in{\partial}\iota_C(x)$, there exists $t\in\mathbb{R}$, such that $v_\Lambda = t x_\Lambda$.
		\item If $r=n$, i.e., $C=\{x\in\mathbb{R}^n:\|x\|_2=1 \}$, then $\widehat{\partial}\iota_C(x) = {\partial}\iota_C(x) = \{v\in\mathbb{R}^n:v=tx,t\in\mathbb{R} \}$.
	\end{enumerate}
\end{lemma}

The proof of Lemma \ref{lemma:7.0-1} is given in Appendix \ref{appendixC}. With the help of Lemma \ref{lemma:7.0-1}, we characterize the relationship between the critical points of $G$ and the generalized eigenvectors of matrix pair $(A,B)$ or the related sub-matrix pair of $(A,B)$.
\begin{proposition}\label{ppsition:7.0-1}
	Let $x^\star\in \dom(G)$ and $\Lambda$ be the support of $x^\star$. Then $x^\star$ is a critical point of $G$ if and only if one of the following statements hold:
	\begin{enumerate}[{\upshape(\romannumeral 1)}]
		\item $|\Lambda| <r$ and $x^\star$ is a unit generalized eigenvector with respect to the generalized eigenvalue $1/G(x^\star)$ of the matrix pair $(A, B)$, i.e., $Bx^\star = G(x^\star) A x^\star$;
		\item $|\Lambda| = r$ and $x^\star_\Lambda$ is a unit generalized eigenvector with respect to the generalized eigenvalue $1/G(x^\star)$ of the matrix pair $(A_\Lambda, B_\Lambda)$, i.e., $B_\Lambda x^\star_\Lambda = G(x^\star) A_\Lambda x^\star_\Lambda$.
	\end{enumerate}
\end{proposition}
\begin{proof}
	According to Definition \ref{Def:critical_point}, $x^\star$ is a critical point of $G$ if and only if
	\begin{equation}\label{formula:7-0ppsition1-1}
	0\in\widehat{\partial}\iota_C(x^\star)+2B x^\star - 2G(x^\star)Ax^\star.
	\end{equation}
	We first prove Item {\upshape(\romannumeral 1)}. Assume that $|\Lambda|<r$. By Lemma \ref{lemma:7.0-1}, the inclusion \eqref{formula:7-0ppsition1-1} is equivalent to the following relation
	\begin{equation}
	d_1 x^\star + 2Bx^\star - 2G(x^\star)Ax^\star = 0
	\end{equation}
	for some $d_1\in\mathbb{R}$. Multiplying $(x^\star)^T$ on both sides of the above equality, we get that $d_1=0$. This proves Item {\upshape(\romannumeral 1)}.
	
	Next, we prove Item {\upshape(\romannumeral 2)}. Suppose that $|\Lambda| = r$. Invoking Lemma \ref{lemma:7.0-1} in this case, inclusion \eqref{formula:7-0ppsition1-1} implies that there exist $d_2\in\mathbb{R}$ and $v\in\mathbb{R}^n$ such that $v_\Lambda = d_2 x^\star_\Lambda$ and
	\begin{equation}\label{formula:ver15:5-2}
	v + 2Bx^\star-2G(x^\star)Ax^\star = 0.
	\end{equation}
	This yields that
	\begin{equation*}
	d_2x^\star_\Lambda + 2B_\Lambda x^\star_\Lambda - 2G(x^\star)A_\Lambda x^\star_\Lambda = 0.
	\end{equation*}
	Multiplying $(x^\star_\Lambda)^T$ on both sides of the above equality, we immediately obtain $d_2 =0$ and
	\begin{equation}\label{formula:ver15:5-1}
	B_{\Lambda}x^{\star}_{\Lambda} = G(x^{\star})A_{\Lambda}x^{\star}_{\Lambda}.
	\end{equation}
	
	Conversely, if $x^{\star}$ satisfies \eqref{formula:ver15:5-1}, set $v\in\mathbb{R}^n$ to be the vector that $v_{\Lambda} = 0$ and $v_{\Lambda^C} = 2(G(x^{\star})Ax^{\star}-Bx^{\star})_{\Lambda^C}$. Then, $v\in\widehat{\partial}\iota_C(x^{\star})$ and \eqref{formula:ver15:5-2} holds, that imply inclusion \eqref{formula:7-0ppsition1-1}. We then complete the proof.
\end{proof}

\subsection{Implementation and convergence of PGSA and PGSA\_L for problem \eqref{117E4}}
In this subsection, we discuss the implementation of PGSA and PGSA\_L for problem \eqref{117E4} and then establish their convergence results.

We note that the proposed algorithms for problem \eqref{117E4} mainly involve the computation of proximity operator associated with $\iota_C$ and the gradients of $x^TAx$ and $x^TBx$. Thus, the computational cost in these algorithms relies heavily on $\mathrm{prox}_{\iota_C}$, which is exactly the projection operator onto $C$, denoted here by $\mathrm{proj}_C$. We next show that $\mathrm{proj}_C$ has a closed form and thus can be efficiently computed. To this end, we first recall
%
%
%
%
the projection operator onto the set $\{x\in\mathbb{R}^n:\|x\|_0\leq r \}$, denoted by $T_r(x)$. It is well-known that for $x\in\mathbb{R}^n$, $(T_r(x))_i=x_i$ for the $r$ largest components in absolute value of $x$ and $(T_r(x))_i=0$ else. Since the $r$ largest components may not be uniquely defined, $T_r$ is a set-valued operator. With the help of $T_r$ and Proposition 4.3 in \cite{Luss-Teboulle:SIAM2013}, we can immediately obtain the closed form of $\mathrm{proj}_C$ in the following proposition.
\begin{proposition}\label{ppsition:7.0-2}
	Given $x\in\mathbb{R}^n$, then
	\begin{equation*}
	\mathrm{proj}_C(x) = \begin{cases}
	\{\frac{y}{\|y\|_2}:y\in T_r(x) \},&\text{if }x\neq 0,\\
	C,&\text{else. }
	\end{cases}
	\end{equation*}
\end{proposition}

Next, we investigate the convergence property of PGSA and PGSA\_L for problem \eqref{117E4} based on the general convergence results presented in Section \ref{ssection3.3:Global_convergence} and Section \ref{section5:PGSA_L}. To this end, we shall verify Assumptions 1-4 hold for problem \eqref{117E4} and $G$ is a KL function. First, since $B_{\Lambda}$ is symmetric positive definite for any subset $\Lambda\subseteq \mathbb{N}_n$ with $|\Lambda| \leq r$, then $\iota_C(x) + x^TBx$ does not attain $0$ for all $x\in\mathbb{R}^n$. Second, the level boundedness of $G$ follows from the boundedness of $C$. In addition, it is obvious that $\iota_C$ is locally Lipschitz continuous on $C$ and $x^TAx$ is continuously differentiable with a Lipschitz continuous gradient. Finally, we show that in the following proposition $G$ is a semialgebraic function and thus satisfies the KL property. We refer readers to \cite[Section 2.2]{Attouch-Bolte:MP:2013} for the definition of the semialgebraic function and its relation to the KL property.

\begin{proposition}\label{lemma:7.1-1-G_is_semialgebraic}
	$G$ is a semialgebraic function.
\end{proposition}
\begin{proof}
	According to the definition of the semialgebraic function, it suffices to show that $\mathrm{Graph}(G)$ is a semialgebraic set. By the definition of $G$ and the positive semidefinite of $A$, we have
	\begin{align*}
	&\mathrm{Graph}(G)\\ &= \left\{(x,s)\in\mathbb{R}^n\times\mathbb{R}:~ \|x\|_2=1,~\|x\|_0\leq r,~\frac{x^TBx}{x^TAx}=s,~x^TAx\neq 0 \right\}\\
	&=\{(x,s)\in\mathbb{R}^n\times\mathbb{R}:~ \|x\|_2^2=1,~\|x\|_0\leq r,~x^TBx-sx^TAx = 0,~x^TAx>0 \}\\
	&= \bigcup_{\substack{\Lambda\subseteq\mathbb{N}_n\\|\Lambda|=n-r}}
	\{(x,s)\in\mathbb{R}^n\times\mathbb{R}: \|x\|_2^2=1,x_\Lambda=0,x^TBx-sx^TAx = 0,-x^TAx<0 \},
	\end{align*}
	which implies that $\mathrm{Graph}(G)$ is a semialgebraic subset of $\mathbb{R}^{n+1}$. This completes the proof.
\end{proof}

Therefore, in view of Theorems \ref{theorem:4.3.1-1}, \ref{theorem:118triangle} and \ref{theorem:118star}, we immediately obtain the following two theorems regarding the convergence of PGSA and PGSA\_L for problem \eqref{117E4}.

\begin{theorem}\label{theorem:7.0-1}
	Let $\{x^k:k\in\mathbb{N}\}$ be generated by PGSA and PGSA\_ML (PGSA\_L with $N=0$) for problem \eqref{117E4}. Then $\{x^k:k\in\mathbb{N}\}$ converges globally to a critical point of $G$.
\end{theorem}

\begin{theorem}
	Let $\{x^k:k\in\mathbb{N}\}$ be generated by PGSA\_NL (PGSA\_L with $N>0$) for problem \eqref{117E4}. Then $\{x^k:k\in\mathbb{N}\}$ is bounded and any of its accumulation points is a critical point of $G$.
\end{theorem}


\subsection{Convergence rate of PGSA and PGSA\_ML for problem \eqref{117E4}}
In this subsection, we consider the convergence rate of $\{x^k:k\in\mathbb{N}\}$ generated by PGSA and PGSA\_ML for problem \eqref{117E4}. By Theorem \ref{theorem:7.0-1}, the sequence $\{x^k:k\in\mathbb{N}\}$ converges to $x^\star$, which is a critical point of $G$. According to Theorems \ref{theorem:4.3.1-2} and \ref{theorem:202 5.4}, we can further estimate the convergence rate of $\{x^k:k\in\mathbb{N}\}$ by showing that $G$ satisfies the KL property at $x^{\star}$ with $\phi(s) = ds^{1-\theta}$ for some $d>0$ and $\theta \in[0,1)$.

To this end, we first prove that the objective function of the generalized eigenvalue problem (without sparsity constraint) satisfies the KL property with the corresponding $\phi(s) = ds^{\frac{1}{2}}$ for some $d>0$ in the following proposition.
\begin{proposition}\label{ppsition:7.0-3}
	Given $D\in \mathbb{S}^m_+$ and $E\in \mathbb{S}^m_{++}$, let $\varphi:\mathbb{R}^m\to\overline{\mathbb{R}}$ be defined at $x\in\mathbb{R}^m$ as
	\begin{equation}\label{formula:7.0-3-1}
	\varphi(x) := \begin{cases}
	\frac{x^TEx}{x^TDx},&\text{if }\|x\|_2=1 \text{ and } x^TDx\neq 0,\\
	+\infty,&\text{else.}
	\end{cases}
	\end{equation}
	Then $\varphi$ satisfies the KL property at any $\hat{x}\in\dom(\varphi)$ with the corresponding $\phi(s) = ds^{\frac{1}{2}}$ for some $d>0$, i.e., there exist $d>0$, $\eta\in(0,+\infty]$ and a neighborhood $U$ of $\hat{x}$, such that for any $x\in U \cap \{x\in\mathbb{R}^m:~\varphi(\hat{x})<\varphi(x)<\varphi(\hat{x})+\eta \}$,
	\begin{equation*}
	\mathrm{dist}(0,{\partial}\varphi(x))\geq\frac{2}{d}\sqrt{\varphi(x)-\varphi(\hat{x})}.
	\end{equation*}
\end{proposition}

\begin{proof}
	Denote by $\lambda_i$ the $i$-th largest eigenvalue of $E^{-1}D$ for $i\in\mathbb{N}_m$. If $\lambda_i\equiv\lambda_1$ for $i\in\mathbb{N}_m$, then it is trivial that $\varphi(x)\equiv 1/\lambda_1$ for $x\in\dom(\varphi)$ and we immediately prove this proposition. Below we assume that $\lambda_i \not\equiv \lambda_1$. By Lemma \ref{lemma:7.0-1} {\upshape(\romannumeral 3)} and the sum rule of subdifferential, we have for any $x\in\dom(\varphi)$ that
	\begin{equation}\label{eq:1214biaohao}
		{\partial}\varphi(x) = \left\{tx+\frac{2Ex-2\varphi(x)Dx}{x^TDx}:~t\in\mathbb{R} \right\}.
	\end{equation}
	In view of Definition \ref{Def:critical_point} with its remark and invoking again Lemma \ref{lemma:7.0-1} {\upshape(\romannumeral 3)}, we see that $x\in\dom(\varphi)$ is a critical point of $\varphi$ if and only if $0\in\partial\varphi(x)$. Then it suffices to prove that $\varphi$ has the KL property with an exponent $\frac{1}{2}$ at any of its critical points, since a proper lower semicontinuous function always satisfies the KL property with an arbitrary exponent in $[0,1)$ at any point where the limiting subdifferential does not contain 0, see, for example, \cite[Lemma 2.1]{Li-Pong:2018Foundations_of_computational_mathematics}.
	
	Let $\hat{x}\in\mathbb{R}^n$ be a critical point of $\varphi$. From Proposition \ref{ppsition:7.0-1}, we have $E^{-1}D\hat{x}=\lambda_j\hat{x}$ and $\lambda_j = 1/\varphi(\hat{x}) > 0$ for some $j\in\mathbb{N}_m$.
	Using the fact that $\langle x,\frac{2Ex-2\varphi(x)Dx}{x^TDx} \rangle = 0$, we deduce from \eqref{eq:1214biaohao} that
	\begin{equation*}
	\mathrm{dist}(0,{\partial}\varphi(x)) = \left\|\frac{2Ex-2\varphi(x)Dx}{x^TDx}\right\|_2.
	\end{equation*}
	Let $U$ be a neighborhood of $\hat{x}$ such that for all $x\in U$, there hold $\frac{1}{2}\hat{x}^TD\hat{x} \leq x^TDx\leq 2\hat{x}^TD\hat{x}$, $\frac{1}{2}\hat{x}^TE\hat{x} \leq x^TE x \leq 2\hat{x}^TE\hat{x}$ and $x^T\hat{x} \neq 0$. Then, for any $x\in U\cap \mathrm{dom}(\varphi)$, it holds that
	\begin{equation}\label{formula:7EQ1}
	\mathrm{dist}(0,{\partial}\varphi(x)) \geq \frac{\sqrt{\mu}}{\hat{x}^T D \hat{x}}\|Ex-\varphi(x)Dx\|_{E^{-1}},
	\end{equation}
	where $\mu>0$ is the smallest eigenvalue of $E$. By a direct computation we have that
	\begin{align}
	\|Ex-\varphi(x)Dx\|^2_{E^{-1}}
	&= (Ex-\varphi(x)Dx)^TE^{-1}(Ex-\varphi(x)Dx)\notag\\
	&= \frac{(x^TEx)^3}{(x^TDx)^2}\left( \frac{x^TDE^{-1}Dx}{x^TEx} - \left(\frac{x^TDx}{x^TEx}\right)^2~ \right)\notag\\
	&\geq \frac{(\hat{x}^TE\hat{x})^3}{32(\hat{x}^TD\hat{x})^2}
	\left( \frac{x^TDE^{-1}Dx}{x^TEx} - \left(\frac{x^TDx}{x^TEx}\right)^2~ \right).\label{formula:7EQ3}
	\end{align}
	
	On the other hand, for $x\in U$ with $\varphi(x)>\varphi(\hat{x})$, we get that
	\begin{align}
	\varphi(x)-\varphi(\hat{x}) &= \frac{x^T E x}{x^T D x} - \frac{1}{\lambda_j} = \frac{x^T E x}{\lambda_j x^T D x}\left(\lambda_j - \frac{x^T D x}{x^T E x}\right) \notag\\
	&\leq \frac{4\hat{x}^T E \hat{x}}{\lambda_j\hat{x}^T D \hat{x}} \left(\lambda_j - \frac{x^T D x}{x^T E x}\right).\label{formula:7EQ4}
	\end{align}
	
	In view of \eqref{formula:7EQ1}, \eqref{formula:7EQ3} and \eqref{formula:7EQ4}, we can obtain the desired result by showing that there exist $d_1$, $\eta>0$ such that
	\begin{equation}\label{formula:7EQ5}
	\frac{x^TDE^{-1}Dx}{x^TEx} - \left(\frac{x^TDx}{x^TEx}\right)^2
	\geq d_1\left(\lambda_j - \frac{x^T D x}{x^T E x}\right)
	\end{equation}
	whenever $x\in U$ and $\varphi(\hat{x}) < \varphi(x) <\varphi(\hat{x})+\eta$. To this end, we first introduce an equivalent formulation of \eqref{formula:7EQ5}. Since $E\in\mathbb{S}^m_{++}$, we know that $E=HH^T$ for some $m\times m$ invertible matrix $H$. The fact $D\in\mathbb{S}^m_+$ indicates that $H^{-1}DH^{-T}\in\mathbb{S}^m_+$ and thus there exists an orthonormal matrix $Q$ such that $H^{-1}D H^{-T}=Q\Sigma Q^T$, where $\Sigma = \mathrm{diag}\{\lambda_1,\lambda_2,\cdots,\lambda_m \}$. Then, a direct computation yields that $x^TEx = \|Q^TH^Tx\|^2_2$, $x^TDx=(Q^TH^Tx)^T\Sigma Q^T H^Tx$ and $x^TDE^{-1}Dx=(Q^TH^Tx)^T \Sigma^2 Q^TH^Tx$. Using the above relations, we deduce that \eqref{formula:7EQ5} is equivalent to
	\begin{equation}\label{formula:7EQ6}
	z^T\Sigma^2z - (z^T\Sigma z)^2 \geq d_1 z^T(\lambda_jI-\Sigma)z,
	\end{equation}
	where $z=Q^TH^Tx / \|Q^TH^Tx\|_2$.
	
	Now it remains to show \eqref{formula:7EQ6}. Let $d_1:=\frac{1}{2}\min\{|\lambda_i-\lambda_j|:\lambda_i\neq \lambda_j \}$ and $\eta:=d_1/\lambda_j^2$. For $x\in\{x\in\mathbb{R}^m:\varphi(\hat{x}) < \varphi(x) <\varphi(\hat{x})+\eta \}$, we see that
	\begin{equation*}
	0<z^T(\lambda_jI-\Sigma)z = \frac{1}{\varphi(\hat{x})} - \frac{1}{\varphi(x)}
	<\frac{\eta}{\varphi^2(\hat{x})} = d_1.
	\end{equation*}
	Using this fact and invoking $\|z\|_2=1$, we further have
	\begin{align*}
	z^T\Sigma^2z - (z^T\Sigma z)^2
	&= z^T(\lambda_jI-\Sigma)^2 z - (z^T(\lambda_jI-\Sigma)z)^2\\
	&\geq z^T(\lambda_jI-\Sigma)^2 z - d_1z^T(\lambda_jI-\Sigma)z\\
	&= z^T(\lambda_jI-\Sigma)((\lambda_j-d_1)I-\Sigma)z\\
	&= \sum_{i=1}^{m}(\lambda_j-\lambda_i)(\lambda_j-\lambda_i-d_1)z^2_i\\
	&\geq d_1\sum_{i=1}^{m}(\lambda_j-\lambda_i) z^2_i\\
	&= d_1 z^T(\lambda_j I-\Sigma)z.
	\end{align*}
	We complete the proof.
\end{proof}

Now, we are ready to prove that $G$ satisfies the KL property with the corresponding $\phi(s) = ds^{\frac{1}{2}}$ for some $d>0$.
\begin{proposition}\label{ppsition:ver16_section5_5.7}
	The function $G$ satisfies the  KL property at any $\tilde{x}\in\dom(G)$ with the corresponding $\phi(s) = ds^{\frac{1}{2}}$ for some $d>0$.
\end{proposition}
\begin{proof}
	Let $\tilde{\Lambda} := \mathrm{supp}(\tilde{x})$ and it is clear that $|\tilde{\Lambda}|\leq r$. Given $\Lambda \subseteq \mathbb{N}_n$, let $\varphi_{\Lambda}$ be the function $\varphi$ which is defined in \eqref{formula:7.0-3-1} with respect to $D=A_{\Lambda}$, $E=B_{\Lambda}$. By Proposition \ref{ppsition:7.0-3}, for any $\Lambda\subseteq\mathbb{N}_n$, there exist $d_{\Lambda}>0$, $\eta_{\Lambda}>0$ and $\delta_{\Lambda}>0$ such that for all $z\in U(\tilde{x}_{\Lambda},\delta_{\Lambda})\cap\{z\in\mathbb{R}^{|\Lambda|}:\varphi_{\Lambda}(\tilde{x}_{\Lambda})<\varphi_{\Lambda}(z)<\varphi_{\Lambda}(\tilde{x}_{\Lambda}) + \eta_{\Lambda} \}$
	\begin{equation*}
	\dist (0,{\partial}\varphi_{\Lambda}(z)) \geq \frac{2}{d_{\Lambda}}\sqrt{\varphi_{\Lambda}(z)-\varphi_{\Lambda}(\tilde{x}_{\Lambda})}.
	\end{equation*}
	
	Let $d:=\max\{d_{\Lambda}:\tilde{\Lambda}\subseteq\Lambda\subseteq\mathbb{N}_n,~|\Lambda|\leq r \}$, $\eta:=\min\{\eta_{\Lambda}:\tilde{\Lambda}\subseteq\Lambda\subseteq\mathbb{N}_n,~|\Lambda|\leq r \}$ and $\delta:=\min\{\delta_1,\delta_2 \}$ with $\delta_1:=\min\{d_{\Lambda}:\tilde{\Lambda}\subseteq\Lambda\subseteq\mathbb{N}_n,~|\Lambda|\leq r \}$ and $\delta_2:=\frac{1}{2}\min\{|\tilde{x}_i|:i\in\tilde{\Lambda} \}$. Take any $x\in U(\tilde{x},\delta)\cap\{x\in\mathbb{R}^n:G(\tilde{x})<G(x)<G(\tilde{x})+\eta \}$ and set $\Lambda:=\supp(x)$. Then we immediately see that $\tilde{\Lambda}\subseteq\Lambda$ with $|\Lambda|\leq r$, $G(x) = \varphi_{\Lambda}(x_{\Lambda})$ and $G(\tilde{x})=\varphi_{\Lambda}(\tilde{x}_{\Lambda})$. In addition, one can check that $x_{\Lambda}\in U(\tilde{x}_{\Lambda},\delta_{\Lambda})\cap\{z\in\mathbb{R}^{|\Lambda|}:\varphi_{\Lambda}(\tilde{x}_{\Lambda})<\varphi_{\Lambda}(z)<\varphi_{\Lambda}(\tilde{x}_{\Lambda}) + \eta_{\Lambda} \}$. Also, by Lemma \ref{lemma:7.0-1}, we have
	\begin{align}\label{formula:7EQ1-2}
	\dist(0,{\partial} G(x))
	&\geq \dist\left( 0,\left\{tx_{\Lambda} + \frac{2B_{\Lambda}x_{\Lambda}-2G(x)A_{\Lambda}x_{\Lambda}}{x^T_{\Lambda}A_{\Lambda}x_{\Lambda}}:t\in\mathbb{R}\right \} \right)\notag\\
	&=\dist(0,{\partial}\varphi_{\Lambda}(x_{\Lambda})).
	\end{align}
	Using the aforementioned facts, we finally have
	\begin{equation*}
	\dist(0,{\partial}G(x))
	\geq \dist(0,{\partial}\varphi_{\Lambda}(x_{\Lambda}))
	\geq\frac{2}{d_{\Lambda}}\sqrt{\varphi_{\Lambda}(x_{\Lambda}) - \varphi_{\Lambda}(\tilde{x}_{\Lambda})}
	\geq\frac{2}{d}\sqrt{G(x)-G(\tilde{x})}.
	\end{equation*}
	This completes the proof.
\end{proof}

With the help of Theorems \ref{theorem:4.3.1-2}, \ref{theorem:202 5.4}, \ref{theorem:7.0-1} and Proposition \ref{ppsition:ver16_section5_5.7}, we immediately establish the main theorem of this subsection regarding the convergence rate of PGSA and PGSA\_ML.
\begin{theorem}\label{theorem:ver16_section5_triangle}
	The sequence $\{x^k:k\in\mathbb{R}\}$ generated by PGSA or PGSA\_ML converges R-linearly to a critical point of $G$.
\end{theorem}	


If the initial point is close enough to a global minimizer of $G$, we further have the following convergence result, concerning PGSA and PGSA\_ML for problem \eqref{117E4}.

\begin{corollary}\label{corollary:202 611}
	Let $\tilde{x}\in\mathbb{R}^n$ be a global minimizer of $G$. Then there exists $\delta >0$, such that the sequence $\{x^k:k\in\mathbb{N} \}$ generated by PGSA or PGSA\_ML for problem \eqref{117E4} with $\|x^0-\tilde{x}\|_2<\delta$ converges R-linearly to a global minimizer of $G$.
\end{corollary}
\begin{proof}
	By Theorem \ref{theorem:7.0-1} and Theorem 2.12 in \cite{Attouch-Bolte:MP:2013}, there exists $\delta>0$, such that $\{x^k:k\in\mathbb{N} \}$, which starts from $x^0$ satisfying $\|x^0-\tilde{x}\|_2<\delta$, converges to a global minimizer $\bar{x}$ of $G$. We then obtain the desired result immediately from Theorem \ref{theorem:ver16_section5_triangle}.
\end{proof}

To close this section, we point out the relation between PGSA for problem \eqref{117E4} and an existing algorithm for SGEP. Very recently, the authors in \cite{Tan-Wang-Liu-Zhang:JRSS2018} propose a truncated Rayleigh flow method (TRFM) for solving SGEP and show that TRFM converges R-linearly to a global minimizer of $G$ when the initial point $x^0$ is close enough to that global minimizer. By appropriate reformulations, we observe that the iteration procedure of TRFM essentially coincides with that of PGSA for problem \eqref{117E4} with a constant step size in $(0,1/L)$. However, there are great differences between the convergence results and the proof of PGSA and TRFM. First, we not only establish the convergence of PGSA with the initial point close to a global minimizer in Corollary \ref{corollary:202 611}, but also prove that PGSA converges R-linearly to a critical point for arbitrary starting points. On the other hand, there is no convergence guarantee in \cite{Tan-Wang-Liu-Zhang:JRSS2018} for TRFM starting from an arbitrary point. Second, our convergence analysis for PGSA is primarily based on the KL property of the objective in problem \eqref{117E4}, while the convergence of TRFM is established using some mathematical tools in probability and statistics.


\section{Numerical experiments}\label{section:Numerical experiments}

In this section, we conduct some numerical experiments to test the efficiency of our proposed algorithms, namely, PGSA, PGSA\_ML and PGSA\_NL. We consider three concrete examples of problem \eqref{problem:root}: the sparse Fisher's discriminant analysis (SFDA), the sparse sliced inverse regression (SSIR) and the $\ell_1/\ell_2$ sparse signal recovery. The first two problems are special cases of SGEP, while the third problem is another application of problem \eqref{problem:root}. All the experiments are conducted in Matlab R2019b on a desktop with an Intel(R) Core(TM) i5-9500 CPU (3.00GHz) and 16GB of RAM.

\subsection{Sparse Fisher's discriminant analysis and sliced inverse regression}\label{ssection:SFDA and SSIR}

In this subsection, we focus on two special instances of SGEP: SFDA and SSIR. We compare the performance of the proposed algorithms with a commonly used algorithm for SGEP, \textit{Iteratively Reweighted Quadratic Minorization} (IRQM) \cite{Song-Babu-Palomar:2014IEEE_Signal_Processing}, which approximates the $\ell_0$-norm by some continuous surrogate functions and solves the approximation problem via a quadratic majorization-minimization approach. Three versions of IRQM, namely, IRQM-log, IRQM-Lp, IRQM-exp are developed in \cite{Song-Babu-Palomar:2014IEEE_Signal_Processing} by using the respectively surrogate functions.

We describe the implementation details of the aforementioned algorithms below. It is clear that the Lipschitz constant $L=\|B\|_2$ in SGEP\footnote{Specifically, we use \texttt{eigs(B,1,'largestabs','IsSymmetricDefinite',1)} to compute $\|B\|_2$.} (by letting $g(x)=\frac{1}{2}x^TAx$ and $h(x)=\frac{1}{2}x^TBx$). For PGSA, we set $\alpha_k \equiv 0.99/L$ for $k\in\mathbb{N}$. For PGSA\_ML and PGSA\_NL, we set $a=10^{-3}$, $\underline{\alpha} = 0.99/L ,~\bar{\alpha} = 10^8$, and $\eta = 0.5$. Also, $N$ is set to be $4$ in PGSA\_NL. In addition, we choose $\alpha_{0,0} = 0.99/L$ and $\alpha_{k,0}$ via formula \eqref{120E1} for $k\in\mathbb{N}$. The Matlab source code of IRQM is available online\footnote{https://github.com/junxiaosong/junxiaosong.github.io/tree/master/code}. Since it corporates a term $\rho\|\cdot\|_0$ for some $\rho>0$ to promote the sparsity rather than directly controlling the sparsity, we use a bisection method to find a proper $\rho$ with which IRQM produces a solution with desirable sparsity after hard-thresholding. For other parameters of IRQM, we simply adopt the suggested setting in \cite[Section {\uppercase\expandafter{\romannumeral 5}}.A]{Song-Babu-Palomar:2014IEEE_Signal_Processing}.

The proposed algorithms are initialized at an $x^0\in\mathbb{R}^n$ with $x^0_i = 1/\sqrt{r}$ for $i\in\mathbb{N}_r$ and $x^0_i=0$ otherwise, while they are all terminated when the number of iterations hits $2n$ or $\|x^k-x^{k-1}\|_2\leq 10^{-6}$. Following \cite{Song-Babu-Palomar:2014IEEE_Signal_Processing}, the initial point $x^0$ in IRQM is chosen randomly with each entry following a standard Gaussian distribution and then normalized such that $(x^0)^T B x^0 = 1$, while it is terminated once the number of iterations exceeds 1000 or the successive changes of the objective are less than $10^{-5}$. We remark that IRQM requires the matrix $B\in\mathbb{S}^n_{++}$ in problem \eqref{117E4}. However, as it will be seen later, the corresponding $B$ of SFDA or SSIR is positive semidefinite but $B\notin\mathbb{S}^n_{++}$. For fair comparison, in the experiments we add $0.5I$ to $B$ so that it turns into positive definite and IRQM can be applied.

First, we consider SFDA. Given $p$ data samples $\{z^1,z^2,\dots,z^p\}$ consisting of two distinct classes with $n$ features, let $\mathcal{I}_k\subseteq \mathbb{N}$ be the index set of samples in the $k$-th class and denote $|\mathcal{I}_k|$ by $p_k$ ($k=1$ or $2$). The within-class and between-class covariance matrices are defined as:
\begin{equation*}
\hat{\Sigma}_{\omega} := \frac{1}{p}\sum_{k=1}^{2}\sum_{i\in \mathcal{I}_k} (z^i-\hat{u}^k)(z^i-\hat{u}^k)^T
\quad \text{ and }\quad
\hat{\Sigma}_b := \frac{1}{p}\sum_{k=1}^{2}p_k\hat{u}^k(\hat{u}^k)^T,
\end{equation*}
where $\hat{u}^k:=\sum_{j\in \mathcal{I}_k}z^j/p_k$ for $k=1,2$. For an integer $r\in[1,n]$, the SFDA seeks a sparse projection vector by solving problem \eqref{117E4} with $A=\hat{\Sigma}_b$ and $B=\hat{\Sigma}_{\omega}$.

In the experiments, we use a simulation setting similar to that of \cite{Tan-Wang-Liu-Zhang:JRSS2018}. The samples of the $k$-th class are randomly generated following a Gaussian distribution with mean $u^k$ and covariance $\Sigma$ for $k=1$ and $2$. We set $u^1=0_n$, $u^2_j = 0.5$ for $j\in\{2,4,\dots, 40  \}$ and $u^2_j = 0$ otherwise. Meanwhile, let $\Sigma$ be a block diagonal matrix with five blocks, each of which is in the dimension $(n/5)\times (n/5)$. The $(j,j')$-th entry of each block takes value $0.8^{|j-j'|}$. We fix $p=1000$, $p_1=p_2=500$ and use different values for $n\in\{1000,1500,2000  \}$, while the sparsity rate $r/n$ is varied from $\{0.05,0.1,0.2\}$ for a fixed $n$. For each $(n,r)$, we generate 100 instances of two-class dataset randomly as described above.

\begin{table}[htbp]
	\centering
	\caption{Computational results for SFDA}
	\begin{tabular}{|c|l|cc|cc|cc|}
		\hline
		\multicolumn{2}{|c|}{\multirow{2}[4]{*}{SFDA results}} & \multicolumn{2}{c|}{$n=1000$} & \multicolumn{2}{c|}{$n=1500$} & \multicolumn{2}{c|}{$n=2000$} \bigstrut\\
		\cline{3-8}    \multicolumn{2}{|c|}{} & \multicolumn{2}{c|}{$t_L=0.01$} & \multicolumn{2}{c|}{$t_L=0.04$} & \multicolumn{2}{c|}{$t_L=0.08$} \bigstrut\\
		\hline
		\multicolumn{1}{|l|}{$r/n$} & Alg.  & Obj   & Time  & Obj   & Time  & Obj   & Time \bigstrut\\
		\hline
		\multirow{6}[2]{*}{0.05} & PGSA  & 0.47  & 0.011  & 0.42  & 0.024  & 0.41  & 0.044  \bigstrut[t]\\
		& PGSA\_ML & 0.43  & 0.007  & 0.41  & 0.016  & 0.39  & 0.030  \\
		& PGSA\_NL & 0.43  & 0.006  & 0.41  & 0.013  & 0.39  & 0.024  \\
		& IRQM-log & 0.44  & 0.494  & 0.42  & 1.079  & 0.41  & 2.153  \\
		& IRQM-Lp & 0.45  & 0.480  & 0.42  & 1.058  & 0.41  & 2.064  \\
		& IRQM-exp & 0.44  & 0.494  & 0.42  & 1.074  & 0.41  & 2.155  \bigstrut[b]\\
		\hline
		\multirow{6}[2]{*}{0.1} & PGSA  & 0.41  & 0.020  & 0.39  & 0.050  & 0.37  & 0.110  \bigstrut[t]\\
		& PGSA\_ML & 0.40  & 0.016  & 0.37  & 0.038  & 0.34  & 0.082  \\
		& PGSA\_NL & 0.40  & 0.014  & 0.37  & 0.031  & 0.34  & 0.064  \\
		& IRQM-log & 0.41  & 0.437  & 0.39  & 0.963  & 0.37  & 1.781  \\
		& IRQM-Lp & 0.41  & 0.422  & 0.39  & 0.920  & 0.37  & 1.707  \\
		& IRQM-exp & 0.41  & 0.440  & 0.39  & 0.964  & 0.37  & 1.790  \bigstrut[b]\\
		\hline
		\multirow{6}[2]{*}{0.2} & PGSA  & 0.38  & 0.045  & 0.35  & 0.136  & 0.32  & 0.314  \bigstrut[t]\\
		& PGSA\_ML & 0.37  & 0.037  & 0.34  & 0.103  & 0.30  & 0.194  \\
		& PGSA\_NL & 0.37  & 0.028  & 0.34  & 0.076  & 0.30  & 0.145  \\
		& IRQM-log & 0.38  & 0.399  & 0.35  & 0.902  & 0.33  & 1.689  \\
		& IRQM-Lp & 0.39  & 0.383  & 0.35  & 0.859  & 0.33  & 1.607  \\
		& IRQM-exp & 0.38  & 0.404  & 0.35  & 0.906  & 0.33  & 1.697  \bigstrut[b]\\
		\hline
	\end{tabular}%
	\label{Table1:exp_main_result}%
\end{table}%

\begin{figure}[H]
	\centering
	\makeatletter\def\@captype{figure}\makeatother
	\subfigure[$n=1000,~r=100$]{\scalebox{0.3}{\includegraphics{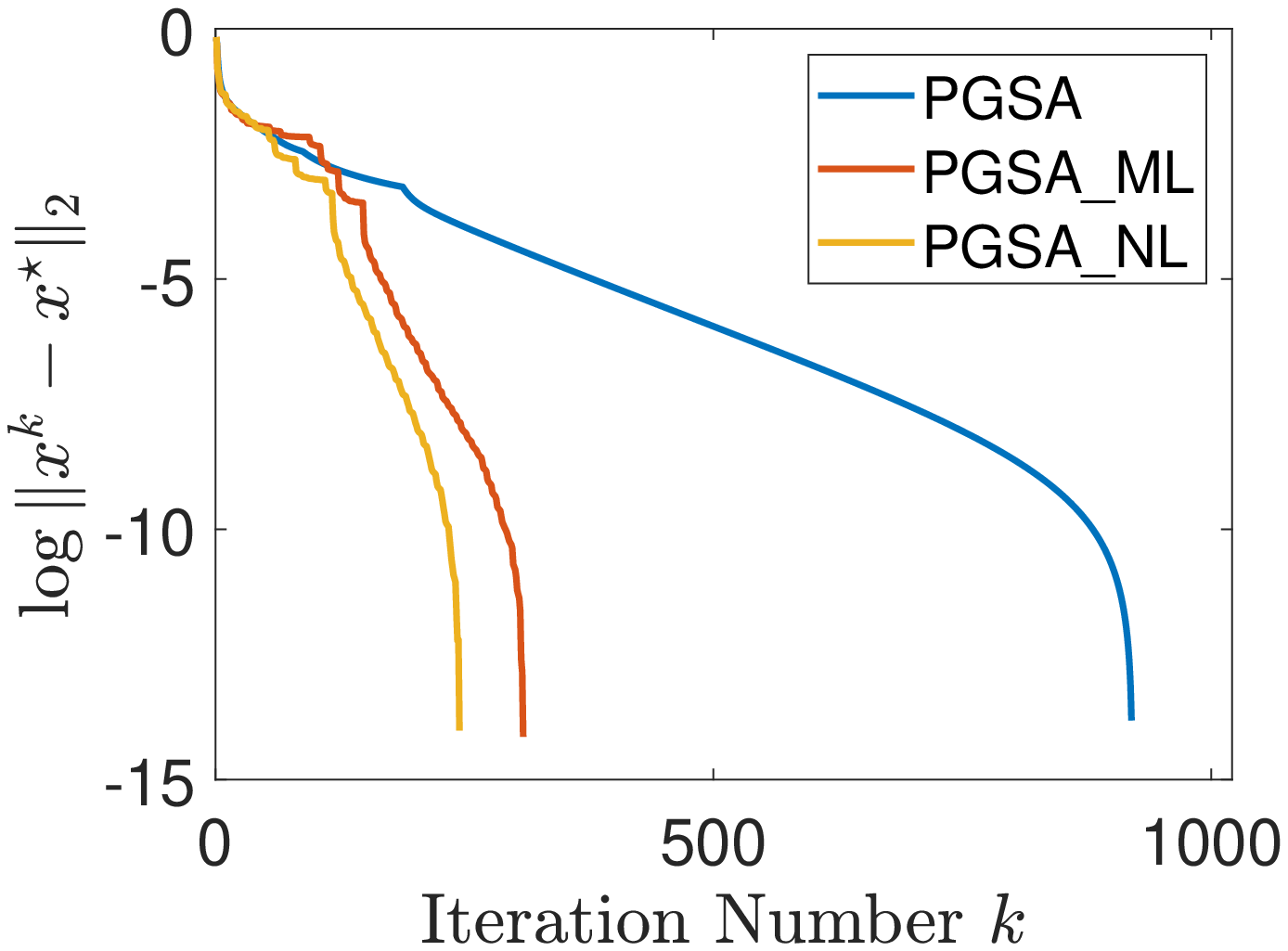}}}
	\subfigure[$n=1500,~r=150$]{\scalebox{0.3}{\includegraphics{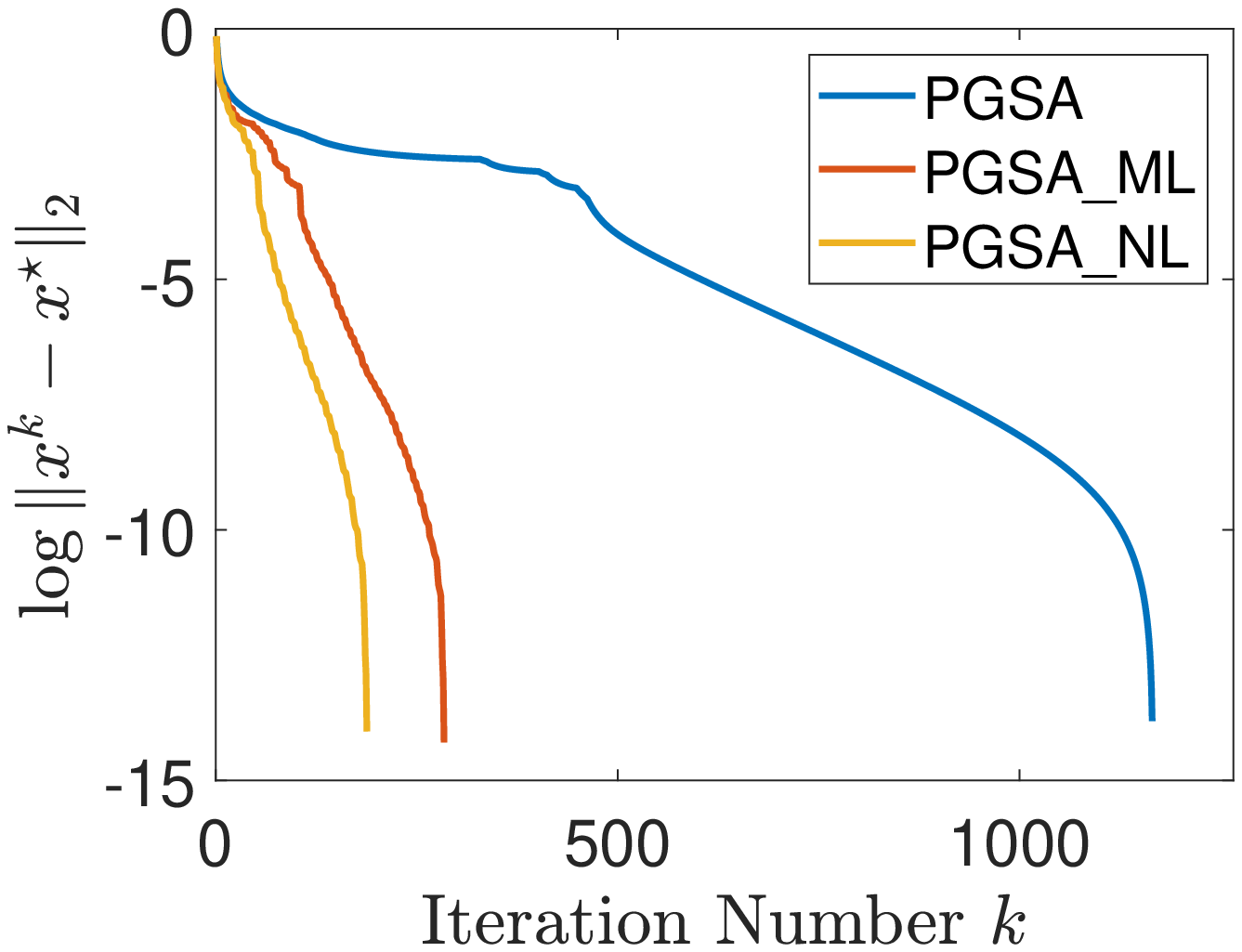}}}
	\subfigure[$n=2000,~r=200$]{\scalebox{0.3}{\includegraphics{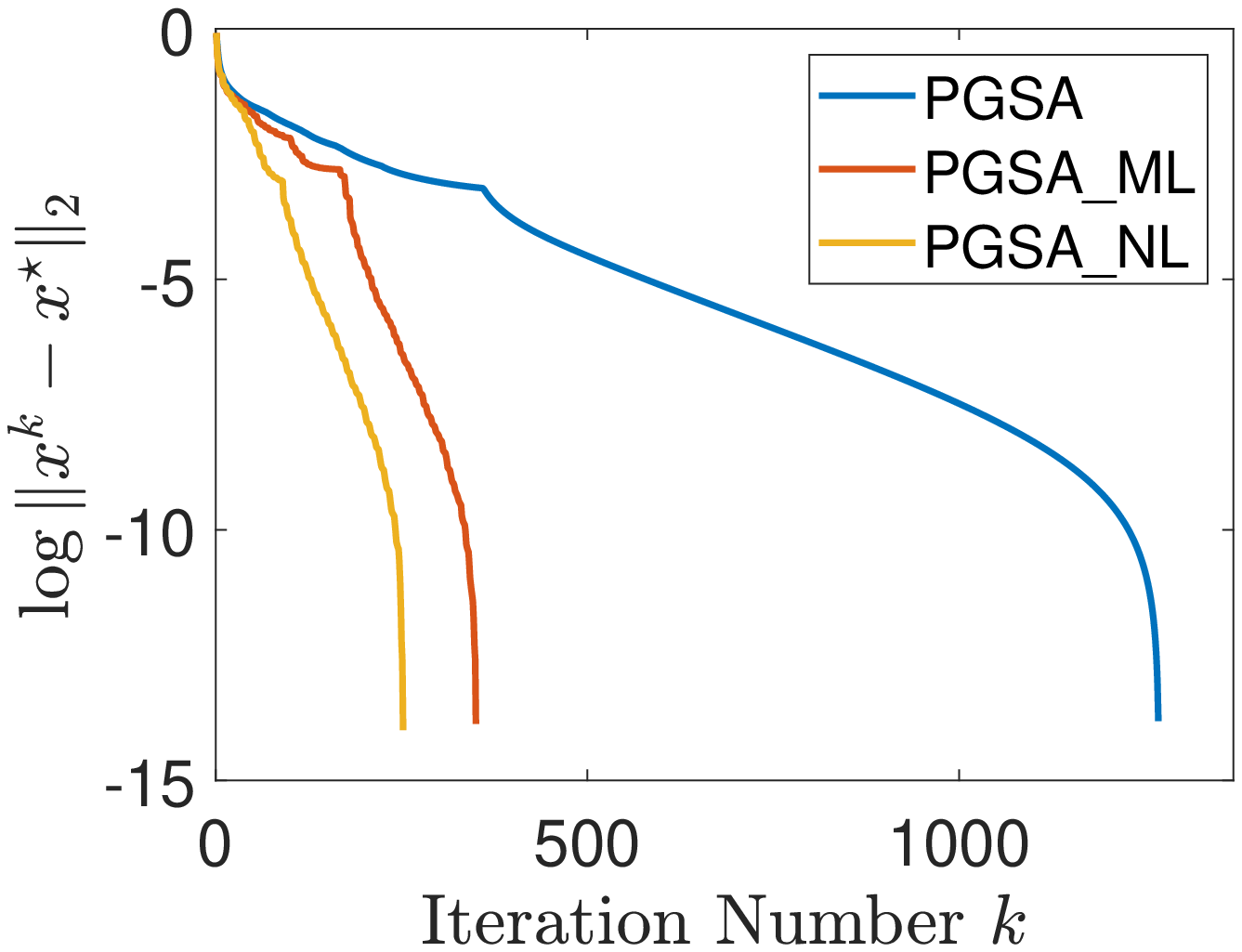}}}
	\caption{Plots of $\|x^k-x^{\star}\|_2$ (in $\log$ scale) for SFDA with different $(n,r)$}
	\label{Figure1:SGEP_SFDA_figure}
\end{figure}

Table \ref{Table1:exp_main_result} reports the computational results averaged over 100 random instances. The two columns for a given $(n,r)$ give the averaged objective value and CPU time (in seconds) of each algorithm. The averaged time $t_L$ of computing $L=\|B\|_2$ is not included in the CPU time column but is reported independently for each dimension $n$. We observe that the proposed algorithms substantially outperform the three IRQM algorithms in terms of CPU time, while the objective values found by the competing algorithms are comparable. In addition, the line-search algorithms PGSA\_ML and PGSA\_NL perform slightly better than PGSA.
Next, we study the convergence rate of the proposed algorithms. In view of Theorem \ref{theorem:ver16_section5_triangle}, one can expect to see R-linear convergence of the sequence generated by PGSA and PGSA\_ML. We plot $\|x^k-x^{\star}\|_2$ (in logarithmic scale) against the number of iterations in Figure \ref{Figure1:SGEP_SFDA_figure}, where $x^{\star}$ is the approximated solution produced by the corresponding algorithm. It is obvious that the sequence generated by PGSA\_ML or PGSA\_NL converges much faster than that by PGSA. As can be seen from Figure \ref{Figure1:SGEP_SFDA_figure}, the sequence generated by PGSA or PGSA\_ML appears to converge R-linearly, which confirms with Theorem \ref{theorem:ver16_section5_triangle}. Finally, we remark that although we have no theoretical results concerning the convergence rate or even convergence of the whole sequence generated by PGSA\_NL, that sequence also seems to converge R-linearly and its convergence rate is slightly faster than that of PGSA\_ML.

Now we consider SSIR for the model $Y = \psi(v_1^TX,...,v_k^TX,\epsilon)$, where $X$ is $n$-dimensional covariates, $Y$ is a univariate response, $\epsilon$ is the stochastic error independent of $X$, and $\psi$ is an unknown link function. Under regularity conditions, the first leading eigenvector of the subspace spanned by $v_1,\cdots,v_k$ can be identified by solving problem \eqref{117E4} with $A=\hat{\Sigma}_{E(X|Y)}$, $B=\hat{\Sigma}_X$, where $\hat{\Sigma}_X$ and $\hat{\Sigma}_{E(X|Y)}$ denote the sample covariance matrix of $X$ and the conditional expectation $E(X|Y)$ respectively. The interested readers can see \cite{Tan-Wang-Liu-Zhang:JRSS2018} and reference therein for more details.

Below we compare the proposed algorithms with IRQM for solving SSIR on 6 real datasets downloaded from scikit-feature selection repository\footnote{https://jundongl.github.io/scikit-feature/datasets.html}, whose characteristics are summarized in Table \ref{Table2:real_dataset}. Also, we set $r = \lceil 0.01 n \rceil$ for each dataset. The computation results are presented in Table \ref{Table2:SSIRexp_main_result}. The objective values and CPU time (in seconds) of the competing algorithms are listed in the two columns for each dataset. Note that the time $t_L$ of computing $L=\|B\|_2$ is not included in the time column but is reported independently for each dataset.

 One can observe that PGSA\_ML and PGSA\_NL significantly outperform the three IRQM algorithms in terms of CPU time. Note that although PGSA substantially outperforms IRQM, it still costs much more CPU time than PGSA\_ML and PGSA\_NL. Since $L$ is large for the real datasets used in this experiment, it is not surprising that PGSA with a small step size $\alpha_k<1/L$ has slower convergence than its line-search counterparts.

\begin{table}[htbp]
	\centering
	\caption{Characteristics of 6 datasets from scikit-feature selection repository}
	\begin{tabular}{|l|ccc|}
		\hline
		Dataset & BASEHOCK & gisette & Prostate\_GE \bigstrut\\
		\hline
		Number of samples $p$ & 1993  & 7000  & 102 \bigstrut[t]\\
		Number of features $n$ & 4862  & 5000  & 5966 \bigstrut[b]\\
		\hline
		\hline
		Dataset & leukemia & ALLAML & arcene \bigstrut\\
		\hline
		Number of samples $p$ & 72    & 72    & 200 \bigstrut[t]\\
		Number of features $n$ & 7070  & 7129  & 10000 \\
		\hline
	\end{tabular}%
	\label{Table2:real_dataset}%
\end{table}%

\begin{table}[htbp]
	\centering
	\caption{Computational results for SSIR}
\begin{tabular}{|l|cc|cc|cc|}
	\hline
	& \multicolumn{2}{c|}{BASEHOCK} & \multicolumn{2}{c|}{gisette} & \multicolumn{2}{c|}{Prostate\_GE} \bigstrut\\
	\cline{2-7}      & \multicolumn{2}{c|}{$t_L=0.16$} & \multicolumn{2}{c|}{$t_L = 0.22$} & \multicolumn{2}{c|}{$t_L=0.23$} \bigstrut\\
	\hline
	Alg.  & Obj   & Time  & Obj   & Time  & Obj   & Time \bigstrut\\
	\hline
	PGSA  & 2.68  & 0.41  & 1.44  & 0.55  & 1.16  & 4.18  \bigstrut[t]\\
	PGSA\_ML & 2.25  & 0.03  & 1.36  & 0.05  & 1.10  & 0.07  \\
	PGSA\_NL & 2.26  & 0.06  & 1.36  & 0.02  & 1.10  & 0.05  \\
	IRQM-log & 2.67  & 6.27  & 1.58  & 6.54  & 1.17  & 10.03  \\
	IRQM-Lp & 2.77  & 5.56  & 1.59  & 6.32  & 1.18  & 9.81  \\
	IRQM-exp & 2.65  & 6.29  & 1.58  & 6.54  & 1.17  & 10.11  \bigstrut[b]\\
	\hline
	\hline
	& \multicolumn{2}{c|}{leukemia} & \multicolumn{2}{c|}{ALLAML} & \multicolumn{2}{c|}{arcene} \bigstrut\\
	\cline{2-7}      & \multicolumn{2}{c|}{$t_L=0.31$} & \multicolumn{2}{c|}{$t_L=0.33$} & \multicolumn{2}{c|}{$t_L=0.65$} \bigstrut\\
	\hline
	Alg.  & Obj   & Time  & Obj   & Time  & Obj   & Time \bigstrut\\
	\hline
	PGSA  & 1.05  & 7.41  & 1.06  & 5.63  & 1.29  & 24.01  \bigstrut[t]\\
	PGSA\_ML & 1.04  & 0.08  & 1.04  & 0.08  & 1.18  & 0.24  \\
	PGSA\_NL & 1.03  & 0.13  & 1.04  & 0.07  & 1.18  & 0.21  \\
	IRQM-log & 1.07  & 14.19  & 1.07  & 16.27  & 1.58  & 46.83  \\
	IRQM-Lp & 1.07  & 13.71  & 1.08  & 16.40  & 1.76  & 41.17  \\
	IRQM-exp & 1.07  & 14.20  & 1.07  & 16.74  & 1.58  & 46.54  \bigstrut[b]\\
	\hline
\end{tabular}%

	\label{Table2:SSIRexp_main_result}%
\end{table}%

To conclude, our experiments for SGEP on both synthetic and real datasets demonstrate the efficiency of the proposed algorithms for solving SGEP.

\subsection{{\boldmath$\ell_1/\ell_2$} sparse signal recovery}

In this subsection, we consider the $\ell_1/\ell_2$ based sparse signal recovery problem, which uses the $\ell_1/\ell_2$ regularization to find a sparse solution of the linear system $Ax=b$, where $A\in\mathbb{R}^{m\times n}$ and $b\in\mathbb{R}^m$ are given. In \cite{Lou-Dong-Wang:2019SIAM}, this problem is formulated into
\begin{equation}\label{problem:L1dL2 box}
\min\left\{\frac{\|x\|_1}{\|x\|_2}:~Ax=b,~\underline{x}\leq x \leq \overline{x},~x\in\mathbb{R}^n \right\},
\end{equation}
where $\underline{x}$, $\overline{x}$ are the lower and upper bounds for the underlying signal. It is not hard to see that problem \eqref{problem:L1dL2 box} is a special case of problem \eqref{problem:root} with $h=0$, $g=\|\cdot\|_2$ and $f=\|\cdot\|_1+\iota_{S_1}$, with $S_1 = \{x\in\mathbb{R}^n:Ax=b,\underline{x}\leq x \leq \overline{x} \}$. Due to $h=0$, PGSA\_ML and PGSA\_NL coincide with PGSA for problem \eqref{problem:L1dL2 box}. In order to apply the line-search scheme, we introduce the following penalty problem of \eqref{problem:L1dL2 box}:
\begin{equation}\label{problem:L1dL2penalty box}
\min\left\{\frac{\lambda \|x\|_1 + \frac{1}{2}\|Ax-b\|^2_2}{\|x\|_2}:~\underline{x}\leq x \leq \overline{x},~x\in\mathbb{R}^n \right\},
\end{equation}
where $\lambda>0$ denotes the penalty parameter. Clearly, problem \eqref{problem:L1dL2penalty box} is also a special instance of problem \eqref{problem:root} with $g=\|\cdot\|_2$, $h=\frac{1}{2}\|A\cdot-b\|^2_2$ and $f=\lambda\|\cdot\|_1+\iota_{S_2}$, where $S_2=\{x\in\mathbb{R}^n: \underline{x}\leq x \leq \overline{x} \}$.

In the experiments, we adopt a simulation setting similar to that of \cite{Lou-Dong-Wang:2019SIAM}. The matrix $A$ is generated by oversampled discrete cosine transformation (DCT), i.e., $A=[a_1,a_2,\cdots,a_n] \in\mathbb{R}^{m\times n}$ with
\begin{equation*}
a_j = \frac{1}{\sqrt{m}}\cos\left( \frac{2\pi w j}{F} \right),~ j=1,2,\cdots,n.
\end{equation*}
Here $w\in\mathbb{R}^m$ is a random vector following the uniform distribution in $[0,1]^m$ and $F>0$ is a parameter measuring how coherent the matrix is. For the ground truth signal $\tilde{x}\in\mathbb{R}^n$, we randomly choose a support set of size $K$ and generate $\tilde{x}$ supported on this set with i.i.d standard Gaussian entries $\mathcal{N}(0,1)$. Then $\tilde{x}$ is normalized to have unit norm and correspondingly we set $\underline{x} = -\mathbf{1}_n$ and $\overline{x} =\mathbf{1}_n$, where $\mathbf{1}_n$ denotes the $n$-dimensional vector with all entries being $1$. Throughout this experiment, we consider the above matrix $A$ of size $(m,n)=(64,1024)$, $F\in\{1,5 \}$ and the ground truth $\tilde{x}$ has sparsity $K=12$.

We consider in the experiments PGSA for problem \eqref{problem:L1dL2 box} and its line-search counterparts for problem \eqref{problem:L1dL2penalty box} with $\lambda=8\times 10^{-5}$ as well as the alternating direction method of multipliers for solving problem \eqref{problem:L1dL2 box} ($L_1/L_2$-ADMM), which is recently proposed in \cite{Lou-Dong-Wang:2019SIAM}.
The implementation details of these algorithms are discussed below. For computing the proximity operator of $\alpha\|\cdot\|_1+\iota_{S_1}$ with $\alpha>0$ in PGSA, we reformulate the related problem into a quadratic programming with linear constraints and then solve it with a commercial software called Gurobi\footnote{https://www.gurobi.com/}. Note that PGSA\_ML and PGSA\_NL both involve the proximity operator of $f=\alpha\lambda\|\cdot\|_1+\iota_{S_2}$ with $\alpha>0$, which has a closed form solution.
Let $z\in\mathbb{R}^n$, one can check that for $j=1,2,\cdots,n$,
\begin{equation*}
(\mathrm{prox}_{\alpha \lambda\|\cdot\|_1+\iota_{S_2}} (z))_j =
\begin{cases}
\underline{x}_j, & \hat{z}_j < \underline{x}_j,\\
\hat{z}_j, & \underline{x} \leq \hat{z}_j \leq \overline{x}_j,\\
\overline{x}_j, & \hat{z}_j > \overline{x}_j,\\
\end{cases}
\end{equation*}
where $\hat{z}_j = \max\{0,|z_j|-\alpha\lambda \}\, \mathrm{sign}(z_j)$. For PGSA\_ML and PGSA\_NL, the parameters are set the same as those in Section \ref{ssection:SFDA and SSIR} except that $L=\|A\|_2^2$ and $\underline{\alpha} = \alpha_{0,0} = 1.99/L$, since $f$ is convex in problem \eqref{problem:L1dL2penalty box}. The Matlab source code for $L_1/L_2$-ADMM is available online\footnote{https://sites.google.com/site/louyifei/Software}. Following the notations in \cite{Lou-Dong-Wang:2019SIAM}, we set the parameters $\rho_1=\rho_2=2000$ for $L_1/L_2$-ADMM.

To obtain an initial point for the competing algorithms, we solve the $\ell_1$-based sparse recovery problem (which replace $\|\cdot\|_1/\|\cdot\|_2$ by $\|\cdot\|_1$ in problem \eqref{problem:L1dL2 box}) by Gurobi. All the algorithms are terminated once the iteration number exceeds $10n=10240$ or $\|x^k-x^{k-1}\|_2 / \|x^k\|_2 \leq 10^{-8}$.

The accuracy of the algorithms is evaluated in terms of success rate, defined as the number of successful trials over the total number of trials. A success is declared when the relative error of the output $x^{\star}$ to the ground truth $\tilde{x}$ is less than $10^{-3}$, that is, $\|x^{\star}-\tilde{x}\|_2/\|\tilde{x}\|_2 < 10^{-3}$. For each $F$, we run all the competing algorithms for 100 trials. Table \ref{Table3: CS_exp_main_result} summarizes the computational results by listing  the value of $\|\cdot\|_1/\|\cdot\|_2$, the averaged CPU time (in seconds) and the success rate for all the algorithms.
The CPU time for computing the initial point is not included in the time column, since all the algorithms use the same initial guess.
We can see the success rate and the value of $\|\cdot\|_1/\|\cdot\|_2$ obtained by PGSA\_ML and PGSA\_NL are comparable to those of PGSA and $L_1/L_2$-ADMM, which are developed for problem \eqref{problem:L1dL2 box}. In terms of CPU time, PGSA\_ML and PGSA\_NL substantially outperform $L_1/L_2$-ADMM, while PGSA performs slightly better than $L_1/L_2$-ADMM. These results demonstrate the efficiency of the proposed algorithms for $\ell_1/\ell_2$ sparse signal recovery.

\begin{table}[htbp]
	\centering
	\caption{Computational results for $\ell_1/\ell_2$ sparse signal recovery}
	\begin{tabular}{|l|ccc|ccc|}
		\hline
		& \multicolumn{3}{c|}{$F=1$} & \multicolumn{3}{c|}{$F=5$} \bigstrut\\
		\hline
		Alg.  & Obj   & Time  & Success & Obj   & Time  & Success \bigstrut\\
		\hline
		$L_1/L_2$-ADMM & 2.845  & 0.212  & 97\%  & 2.852  & 0.278  & 86\% \bigstrut[t]\\
		PGSA  & 2.843  & 0.121  & 97\%  & 2.850  & 0.148  & 86\% \\
		PGSA\_ML & 2.845  & 0.052  & 97\%  & 2.854  & 0.064  & 86\% \\
		PGSA\_NL & 2.844  & 0.048  & 97\%  & 2.854  & 0.055  & 86\% \bigstrut[b]\\
		\hline
	\end{tabular}%
	\label{Table3: CS_exp_main_result}%
\end{table}%

\section{Conclusion}
In this paper, we study a class of single-ratio fractional optimization problems that appears frequently in applications. The numerator of the objective is the sum of a nonsmooth nonconvex function and a nonconvex smooth function, while the denominator is a nonsmooth convex function. We derive a first-order necessary optimality condition for this problem and develop for it first-order algorithms, namely, PGSA, PGSA\_ML and PGSA\_NL. We show the subsequential convergence of the sequence generated by the proposed algorithms under mild assumptions. Moreover, we establish global convergence of the whole sequence generated by PGSA or PGSA\_ML and estimate the convergence rate by additional assumptions on the objective. The proposed algorithms are further applied to solving the sparse generalized eigenvalue problems and their convergence results for the problem are gained according to the general convergence theorems for them. Finally, we conduct some preliminary numerical experiments to illustrate the efficiency of the proposed algorithms.

\appendix
\section{Proof of Proposition \ref{ppsition:2.2-1}}\label{appendixA:proof_pposition2.2}
\begin{proof}
    First we consider the case where $x$ is an isolated point of $\dom(\rho)$. Since $a_2=f_2(x)>0$ and $f_2$ satisfies the calmness condition at x, we deduce that $x$ is an also an isolated point of $\dom(f_1)$. Hence, in this case it is trivial that $\widehat{\partial} \rho(x) = \widehat{\partial}(a_2 f_1-a_1 f_2)=\mathbb{R}^n$. Next we consider the case where $x$ is not an isolated point of $\dom(\rho)$. For any $u\in\dom(\rho)$ and $v\in\mathbb{R}^n$, a direct computation yields
	\begin{equation*}
	\frac{\frac{f_1(u)}{f_2(u)}-\frac{a_1}{a_2}-\langle v,u-x\rangle}{\|u-x\|_2}
	=
	\frac{a_2 f_1(u)-a_1 f_2(u) - \langle a_2^2 v,u-x\rangle}{a_2^2\|u-x\|_2}+R(x,u),
	\end{equation*}
	where $R(x,u) =(a_2-f_2(u))(a_2f_1(u)-a_1 f_2(u))/(a^2_2 f_2(u)\|u-x\|_2) $. Since $f_2$ satisfies the calmness condition and $f_1$ is continuous at $x$, we get that $\lim\limits_{\substack{u\to x\\u\in\dom(f_1)}} R(x,u) = 0$. Noting this fact and by the definition of Fr{\'e}chet subdifferential, we have
	\begin{align*}
	\widehat{\partial} \rho(x)
	&= \left\{v\in\mathbb{R}^n:~\mathop{\lim\inf}\limits_{\substack{u\to x\\u\neq x\\u\in\dom(\rho)}}\,\frac{\frac{f_1(u)}{f_2(u)}-\frac{a_1}{a_2}-\langle v,u-x\rangle}{\|u-x\|_2}\geq 0\right\}\\
	&= \left\{v\in\mathbb{R}^n:~\mathop{\lim\inf}\limits_{\substack{u\to x\\u\neq x\\u\in\dom(f_1)}} \,\frac{a_2 f_1(u)-a_1 f_2(u) - \langle a_2^2 v,u-x\rangle}{a_2^2\|u-x\|_2}\geq 0\right\}\\[8pt]
	&= \frac{\widehat{\partial}(a_2 f_1-a_1 f_2)(x)}{a_2^2}.
	\end{align*}
	We complete the proof.
\end{proof}

\section{Proof of Proposition \ref{2.2ppsition}}\label{appendixA2:proof_pposition3.1}
\begin{proof}
	We only need to prove the proposition holds for local minimizers, since the conclusion for global minimizers can be proven similarly.
	
	Suppose $x^\star$ is a local minimizer of problem \eqref{problem:root}. 
	Then, there exists $\delta >0$ such that for any $x \in B(x^{\star},\delta)\cap\dom(F)$, there holds
	\begin{equation}\label{formula:3.1ppsition1-1}
	0 \leq \frac{f(x)+h(x)}{g(x)} - \frac{f(x^\star)+h(x^\star)}{g(x^\star)}.
	\end{equation}
	This indicates that
	\begin{equation}\label{formula:3.0ppsition1-2}
	0 \leq f(x)+h(x)-\frac{f(x^\star)+h(x^\star)}{g(x^\star)}g(x) = f(x) + h(x) - c_\star g(x)
	\end{equation}
	for all $x \in B(x^{\star},\delta)\cap\dom(F)$, since $g(x)>0$. Due to the fact that the objective function value of problem \eqref{problem:root2} at $x^\star$ is 0, we have that $x^\star$ is a local minimizer of problem \eqref{problem:root2}.
	
	Conversely, if $x^\star$ is a local minimizer of problem \eqref{problem:root2}, then \eqref{formula:3.0ppsition1-2} holds for $x \in B(x^{\star},\delta)\cap\dom(F)$ with some $\delta>0$. By simple calculation, we obtain that \eqref{formula:3.1ppsition1-1} holds for $x \in B(x^{\star},\delta)\cap\dom(F)$. This implies that $x^\star$ is a local minimizer of problem \eqref{problem:root}. We then complete the proof.
\end{proof}

\section{Proof of Lemma \ref{lemma:7.0-1}}\label{appendixC}
\begin{proof}
	By the definition of Fr{\'e}chet subdifferential, we have that
	\begin{equation*}
	\widehat{\partial} \iota_C(x) = \left\{v\in\mathbb{R}^n:~ \mathop{\lim\inf}\limits_{\substack{y\to x\\y\neq x\\y\in C}}\frac{\langle v,x-y\rangle}{\|x-y\|_2}\geq 0 \right\}.
	\end{equation*}
	Let $\Lambda := \supp (x)$. We first prove Item {\upshape(\romannumeral 1)}. In the case that $|\Lambda| = r$, there exists a neighborhood $U$ of $x$, such that $\supp (y)=\Lambda$ for all $y\in U \cap C$. Thus, we obtain that
	\begin{equation*}
	\widehat{\partial}\iota_C(x)
	=\left\{ v\in\mathbb{R}^n:~\mathop{\lim\inf}\limits_{\substack{y_{\Lambda}\to x_{\Lambda}\\y_{\Lambda}\neq x_{\Lambda}\\\|y_{\Lambda}\|_2=1}} \frac{\langle v_{\Lambda},x_{\Lambda}-y_{\Lambda}\rangle}{\|x_{\Lambda}-y_{\Lambda}\|_2}\geq 0 \right\}
	=\{ v\in\mathbb{R}^n:~v_{\Lambda} = tx_{\Lambda},~t\in\mathbb{R} \}.
	\end{equation*}
	
	Next we consider the case when $|\Lambda|<r$. For any $t\in\mathbb{R}$, we have
	\begin{equation*}
	\lim\limits_{\substack{y\to x\\y\in C}}\left|\frac{\langle tx,x-y\rangle}{\|x-y\|_2}\right|
	=\lim\limits_{\substack{y\to x\\y\in C}}\frac{t\|x\|^2_2-tx^Ty}{\sqrt{\|x\|^2_2+\|y\|^2_2-2x^Ty}}
	=\lim\limits_{\substack{y\to x\\y\in C}} \frac{t(1-x^Ty)}{\sqrt{2(1-x^Ty)}}
	=0.
	\end{equation*}
	Hence, we see that $\{v\in\mathbb{R}^n:v=tx,t\in\mathbb{R} \}\subseteq \widehat{\partial} \iota_C(x)$. We further note that for any $v\in\widehat{\partial}\iota_C(x)$,
	\begin{equation*}
	0\leq
	\mathop{\lim\inf}\limits_{\substack{y\to x\\y\neq x\\y\in C}}\frac{\langle v,x-y\rangle}{\|x-y\|_2}
	\leq
	\mathop{\lim\inf}\limits_{\substack{y_{\Lambda}\to x_{\Lambda}\\y_{\Lambda}\neq x_{\Lambda}\\\|y_{\Lambda}\|_2=1}}
	\frac{\langle v_{\Lambda},x_{\Lambda}-y_{\Lambda}\rangle}{\|x_{\Lambda}-y_{\Lambda}\|_2},
	\end{equation*}
	which indicates that $v_{\Lambda} = tx_{\Lambda}$ for some $t\in\mathbb{R}$. Finally, we show that for all $v\in\widehat{\partial} \iota_C(x)$, $v_j=0$ if $j\notin \Lambda$. Otherwise, there exists $\tilde{v}\in\widehat{\partial}\iota_C(x)$ and $j_0\notin \Lambda$ such that $\tilde{v}_{j_0}\neq 0$. Choose $\{y^k:k\in\mathbb{N} \}$ such that $y^k_{\Lambda} = \sqrt{1-1/k^2}x_{\Lambda}$, $y^k_{j_0} = v_j/(k|v_j|)$, and $y^k_j=0$ for $j\notin\Lambda\cup\{j_0\}$. Then we have that $\{y^k:k\in\mathbb{N} \}\subseteq C$ and $\lim\limits_{k\to\infty}y^k=x$. One can verify that $\lim\limits_{k\to\infty}\innerP{\tilde{v}}{x-y}/\|x-y\|_2=-|v_{j_0}|<0$, which contradicts $\tilde{v}\in\widehat{\partial}\iota_C(x)$. This proves Item {\upshape(\romannumeral 1)}.
	
	We turn to Item {\upshape(\romannumeral 2)}. Take any $v\in\partial\iota_C(x)$. By the definition of limiting-subdifferential, there exist $x^k\in C$ and $v^k\in\widehat{\partial}\iota_C(x^k)$ for $k\in\mathbb{N}$, such that $\lim\limits_{k\to\infty} x^k=x$ and $\lim\limits_{k\to\infty} v^k=v$. Hence, we deduce that $\Lambda \subseteq\supp (x^k)$ when $k\geq K$ for some $K\in\mathbb{N}$. Invoking Item {\upshape(\romannumeral 1)}, there exists $\{t_k\in\mathbb{R}:k\geq K \}$ such that $v^k_{\Lambda} = t_kx^k_{\Lambda}$ for $k\geq K$. Let $i_0\in\Lambda$. Then we have $\lim\limits_{k\to\infty} t_k  = \lim\limits_{k\to\infty}\frac{v_{i_0}^k}{x_{i_0}^k} = \frac{v_{i_0}}{x_{i_0}}$. Therefore, we obtain $v_{\Lambda} = \lim\limits_{k\to\infty} t_k x^k_{\Lambda} = \frac{v_{i_0}}{x_{i_0}}x_{\Lambda}$. This complete the proof.
	
	Finally we prove that Item {\upshape(\romannumeral 3)}. When $r=n$, one can easily deduce that $\widehat{\partial}\iota_C(x) = \{v\in\mathbb{R}^n:v=tx,t\in\mathbb{R} \}$ for $x\in C$ from Item {\upshape(\romannumeral 1)}. Take any $v\in\partial\iota_C(x)$. Following a similar argument to proving Item {\upshape(\romannumeral 2)}, we can show that $v=tx$ for some $t\in\mathbb{R}$. This together with $\widehat{\partial}\iota_C(x)\subseteq\partial\iota_C(x)$ yields Item {\upshape(\romannumeral 3)}.	
\end{proof}

\bibliographystyle{siam}


\end{document}